\documentclass[12pt]{amsart}
\usepackage{amssymb,amsmath,amsfonts,latexsym}
\usepackage{bm, enumerate}
\usepackage{xcolor}
\newcommand{\newsection}[1]{\setcounter{equation}{0} \section{#1}}
\newcommand{\bea}{\begin{eqnarray}}
\newcommand{\eea}{\end{eqnarray}}

\newcommand{\clb}{\mathcal{B}}

\newcommand{\cld}{\mathcal{D}}
\newcommand{\cle}{\mathcal{E}}
\newcommand{\clf}{\mathcal{F}}

\newcommand{\clh}{\mathcal{H}}
\newcommand{\clk}{\mathcal{K}}

\newcommand{\clm}{\mathcal{M}}

\newcommand{\clo}{\mathcal{O}}

\newcommand{\clq}{\mathcal{Q}}

\newcommand{\cls}{\mathcal{S}}

\newcommand{\clw}{\mathcal{W}}

\newcommand\Tt{\tilde{T}}
\newcommand\Zp{\mathbb{Z}_+^n}
\newcommand\bH{\mathbb{H}}
\newcommand\bk{\bm{k}}
\newcommand\bl{\bm{l}}

\newcommand{\D}{\mathbb{D}}

\newcommand{\C}{\mathbb{C}}
\newcommand{\z}{\bm{z}}
\newcommand{\w}{\bm{w}}

\newcommand{\raro}{\rightarrow}

\def \qed {\hfill \vrule height6pt width 6pt depth 0pt}
\def\textmatrix#1&#2\\#3&#4\\{\bigl({#1 \atop #3}\ {#2 \atop #4}\bigr)}
\def\dispmatrix#1&#2\\#3&#4\\{\left({#1 \atop #3}\ {#2 \atop #4}\right)}
\newcommand{\be}{\begin{equation}}
\newcommand{\ee}{\end{equation}}
\newcommand{\ben}{\begin{eqnarray*}}
\newcommand{\een}{\end{eqnarray*}}

\newcommand{\NI}{\noindent}

\newcommand{\bi}{\begin{itemize}}
\newcommand{\ei}{\end{itemize}}

\newcommand\la{{\langle }}
\newcommand\ra{{\rangle}}

\newcommand{\dmt}{D_{m,T^*}}
\newcommand{\cmt}{C_{m,T}}

\newcommand{\lnt}{l^2(\mathbb{Z}_+^n, \cld_{m,T^*})}

\newcommand{\arv}{H^2_n}
\newcommand{\wbg}{\bH^2_n}

\newtheorem{Theorem}{\sc Theorem}[section]
\newtheorem{Lemma}[Theorem]{\sc Lemma}
\newtheorem{Proposition}[Theorem]{\sc Proposition}
\newtheorem{Corollary}[Theorem]{\sc Corollary}
\newtheorem{Definition}[Theorem]{\sc Definition}
\newtheorem{Example}[Theorem]{\sc Example}
\newtheorem{Remark}[Theorem]{\sc Remark}

\newtheorem{Note}[Theorem]{\sc Note}
\newtheorem{Question}{\sc Question}
\newtheorem{ass}[Theorem]{\sc Assumption}
\newcommand{\bt}{\begin{Theorem}}
\def\beginlem{\begin{Lemma}}
\def\beginprop{\begin{Proposition}}
\def\begincor{\begin{Corollary}}
\def\begindef{\begin{Definition}}
\def\beginexamp{\begin{Example}}
\def\beginrem{\begin{Remark}}
\def\beginq{\begin{Question}}
\def\beginass{\begin{ass}}
\def\beginnote{\begin{Note}}
\newcommand{\et}{\end{Theorem}}
\def\endlem{\end{Lemma}}
\def\endprop{\end{Proposition}}
\def\endcor{\end{Corollary}}
\def\enddef{\end{Definition}}
\def\endexamp{\end{Example}}
\def\endrem{\end{Remark}}
\def\endq{\end{Question}}
\def\endass{\end{ass}}
\def\endnote{\end{Note}}

\begin{document}

\title[Hypercontractions and factorizations of multipliers]{Hypercontractions and factorizations of multipliers in one and several variables}

\author[Bhattacharjee]{Monojit Bhattacharjee}
\address{Department of Mathematics, Indian Institute of Technology Bombay, Powai, Mumbai, 400076,
India}
\email{monojit12@math.iisc.ernet.in, monojit.hcu@gmail.com}

\author[Das] {B. Krishna Das}
\address{Department of Mathematics, Indian Institute of Technology Bombay, Powai, Mumbai, 400076,
India}
\email{dasb@math.iitb.ac.in, bata436@gmail.com}

\author[Sarkar]{Jaydeb Sarkar}
\address{Indian Statistical Institute, Statistics and Mathematics Unit, 8th Mile, Mysore Road, Bangalore, 560059,
India}
\email{jay@isibang.ac.in, jaydeb@gmail.com}

%\today



\begin{abstract}
We introduce the notion of characteristic functions
for commuting tuples of hypercontractions on Hilbert spaces,
as a generalization of the notion of Sz.-Nagy and Foias
characteristic functions of contractions.
We present an explicit method to compute characteristic functions
of hypercontractions and relate characteristic functions
by means of the factors of Schur-Agler class of functions
and universal multipliers on the unit ball in $\mathbb{C}^n$.
We also offer some factorization properties
of multipliers. Characteristic functions of hypercontrctions
are complete unitary invariant. The Drury-Arveson space and the weighted
Bergman spaces on the unit ball continues to play a significant role
in our consideration. Our results are new even in the special case of single hypercontractions.
\end{abstract}

\maketitle

\newsection{Introduction}

One of the important aspects of the classical Sz.-Nagy and Foias
theory \cite{NF} is that in order to understand non-self adjoint bounded
linear operators on Hilbert spaces, one should also study (analytic)
function theory. For instance, if $T$ is a pure contraction on a Hilbert
space $\clh$ (that is, $\|T h\|_{\clh} \leq \|h\|_{\clh}$ and $\|T^{*q} h\| \raro 0$ as $q \raro \infty$ and for all $h \in
\clh$), then there exist a (coefficient) Hilbert space $\cle$ and an $M_z^*$-invariant closed
subspace $\clq$ (\textit{model space}) of $H^2_{\cle}(\D)$ such that $T^*$ and $M_z^*|_{\clq}$ are unitarily equivalent. Here $M_z$ is the multiplication operator by the coordinate function $z$ (or, \textit{shift}) on the \textit{$\cle$-valued Hardy space} $H^2_{\cle}(\D)$ over the open unit disc $\D$. Moreover, $\clq$ is uniquely determined by the \textit{characteristic function} of $T$ in an appropriate sense. The Sz.-Nagy and Foias characteristic function of a contraction is a canonical operator-valued analytic function on $\D$ and a complete unitary invariant.

This says, on the one hand, pure contractions on Hilbert spaces dilates to shifts on vector-valued Hardy spaces over the unit disc, and on the other hand, the model spaces (as Hilbert subspaces of
vector-valued Hardy spaces) are explicitly and uniquely determined
by characteristic functions.

In this context, it should be remembered that the concept of
Sz.-Nagy and Foias ``dilations and analytic model theory'', as
above, is most useful in operator theory having important
applications in various fields. This has had an enormous influence
on the development of operator theory and function theory in one and
several variables. Needless to say, one goal of multivariable
operator theory and function theory of several complex variables is to
examine whether commuting tuples of contractions on Hilbert spaces admit analytic models as nice as
Sz.-Nagy and Foias analytic models for contractions.

Following Sz.-Nagy and Foias, Agler, in his seminal
papers \cite{JA1, JA2}, introduced a dilation theory for
$m$-hypercontractions: A pure $m$-hypercontraction dilates to shift
on a vector-valued $m$-weighted Bergman space over the unit disc in
$\mathbb{C}$. Agler's idea was further generalized by Muller and
Vasilescu \cite{MV} to commuting tuples of operators: A pure
$n$-tuple of $m$-hypercontraction dilates to $n$-shifts on a
vector-valued $m$-weighted Bergman space over the unit ball in
$\mathbb{C}^n$ (see Section \ref{sec-prel} for more details).

This paper concerns a complete unitary invariant, namely
characteristic functions, one of the basic questions which center
around the Agler, and Muller and Vasilescu's dilation theory, of
commuting tuples of pure $m$-hypercontractions on Hilbert spaces.

The problem of characteristic functions for hypercontractions and
wandering subspaces of shift invariant subspaces of weighted Bergman
spaces in one-variable goes back to Olofsson \cite{AO2, AO1} (also see Ball and Bolotnikov \cite{BV1}). Then
in \cite{ES}, Eschmeier examined Olofsson's approach in several
variables (also see Popescu \cite{GP}). However, Eschmeier's
approach to characteristic functions appears to be more abstract
than the familiar characteristic functions of single contractions or row contractions \cite{BES}.

Here we take a completely different approach to this problem. Namely,
to each pure $m$-hypercontraction on a Hilbert space, we associate
a canonical triple consisting of a Hilbert space and two bounded linear operators,
and refer to this triple as a characteristic triple of the pure $m$-hypercontraction.
The characteristic function of a pure $m$-hypercontraction,
completely determined by a characteristic triple,
is an operator-valued analytic function on the open unit ball in $\mathbb{C}^n$.
Characteristic triple of a pure $m$-hypercontraction
is unique up to unitary equivalence (in an appropriate sense),
which also yields that the characteristic function
is a complete unitary invariant.
We prove that the joint invariant subspaces of a pure $m$-hypercontraction
is completely determined by the factors of the characteristic function.
Unlike the case of $1$-hypercontractions (or row contractions) \cite{BES},
the characteristic function of a pure $m$-hypercontraction
does not admit a transfer function realization.
However, we prove that the characteristic function of pure $m$-hypercontraction
can be (canonically) represented as a product of a universal multiplier
and a transfer function (or a Drury-Arveson multiplier).
 This result is a byproduct of a general factorization theorem
 for contractive multipliers from vector-valued Drury-Arveson spaces
 to a class of reproducing kernel Hilbert spaces on $\mathbb{B}^n$.
 The general factorization theorem for contractive multipliers
 also yields parametrizations of wandering subspaces of the
 joint shift invariant subspaces of reproducing kernel Hilbert spaces.

The results and the method we introduce here seems to be new even in the single hypercontractions case.

We now describe our main results more precisely. Let $m$ and $n$ be natural numbers,
$\mathbb{Z}_+^n$ be the set of $n$-tuples of non-negative integers, that is
\[
\Zp = \{\bk = (k_1, \ldots, k_n): k_1, \ldots, k_n \in \mathbb{Z}_+\},
\]
and let
\[
\mathbb{B}^n = \{\z = (z_1, \ldots, z_n) \in \C^n: \sum_{i=1}^n |z_i|^2 < 1\},
\]
the open unit ball in $\C^n$. We denote by $\clh$, $\clk$, $\cle$ etc.
as separable Hilbert spaces over $\C$, and by $\clb(\clh)$ the set of
all bounded linear operators on $\clh$.

\NI Unless otherwise stated, $T$ will always mean a commuting $n$-tuple
of bounded linear operators $\{T_1, \ldots, T_n\}$ on some Hilbert
space $\clh$. We also adopt the following notations:
\[
T^{\bk} = T_1^{k_1} \cdots T_n^{k_n} \quad \mbox{and} \quad T^{*
\bk} = T_1^{* k_1} \cdots T_n^{* k_n},
\]
and the multinomial coefficient $\rho_m(\bk)$  as

\begin{equation}\label{eq-m coeff}
\rho_m(\bk) = \frac{(m + |\bk|-1)!}{\bk! (m - 1)!},
\end{equation}
and
\begin{equation}\label{eq-rho0k}
\rho_{0}(\bk) =
\begin{cases} 1 & \mbox{if}~ \bk = \bm{0}
\\
0 & \mbox{otherwise}, \end{cases}
\end{equation}
for all $\bk \in \Zp$. We say that $T$ is a \textit{row contraction} if the
row operator $(T_1, \ldots, T_n) : \clh^n \raro \clh$, denoted again
by $T$ and defined by
\[
T(h_1,\dots,h_n) = \sum_{i=1}^n T_ih_i \quad \quad (h_i \in \clh),
\]
is a contraction. More generally,
if we define the completely positive map $\sigma_T :
\clb(\clh) \raro \clb(\clh)$ by
\[
\sigma_T(X) = \sum_{i=1}^n T_i X T_i^* \quad \quad (X \in
\clb(\clh)),
\]
%which is a completely positive map on $\clb(\clh)$,
then $T$ is said
to be an \textit{$m$-hypercontraction} if
\[
(I_{\clb(\clh)} - \sigma_T)^p (I_{\clh}) \geq 0,
\]
for $p = 1, m$. Note that $T$ is an $1$-contraction if and only
if $T$ is a \textit{row contraction} (cf. \cite{WA}). It is now immediate that
\begin{equation}\label{eq-sigmaT}
(I_{\clb(\clh)} - \sigma_T)^p (I_{\clh}) = \sum_{j=0}^p (-1)^j
\binom{p}{j} \sum_{|\bk|=j} \rho_1(\bk) T^{\bk} T^{*\bk}.
\end{equation}
With this notation we get the following interpretation of
hypercontractions: $T$ is an $m$-hypercontraction if and only if $T$
is a row contraction (that is, $(I_{\clb(\clh)} - \sigma_T)
(I_{\clh}) \geq 0$) and $(I_{\clb(\clh)} - \sigma_T)^m (I_{\clh})
\geq 0$. For each $m$-hypercontraction $T$ on $\clh$, we set the
\textit{defect operator} $D_{m, T^*}$ as
\[
D_{m,T^*} = \big[(I_{\clb(\clh)} - \sigma_T)^m
(I_{\clh})\big]^{\frac{1}{2}},
\]
and the \textit{defect space} $\cld_{m, T^*}$ as
\[
\cld_{m, T^*} = \overline{\mbox{ran}} D_{m, T^*}.
\]
An $m$-hypercontraction $T$ is said to be \textit{pure} (cf. \cite{ES, MV, GP}) if
the strong operator limit of $\sigma_T^p(I_{\clh})$ is $0$ as $p\to\infty$.

Now let $T = (T_1, \ldots, T_n)$ be a commuting $n$-tuple of pure $m$-hypercontraction on a Hilbert space
$\clh$. After reviewing the basic definitions and results of the theory of
$m$-hypercontractions in Section \ref{sec-prel},
we prove the existence of a canonical contraction $\cmt: \clh \to \lnt$, a Hilbert space $\cle$,
and bounded linear operators $B \in \clb(\cle, \clh^n)$ and $D \in
\clb(\cle, l^2(\Zp, \cld_{m, T^*}))$ such that the operator matrix
\[
U = \begin{bmatrix} T^* & B \cr \cmt  & D \cr
\end{bmatrix} : \clh \oplus \cle \to
\clh^n \oplus \lnt,
\]
is unitary (see Theorem \ref{thm-ch triple}).

\NI The triple $(\cle, B, D)$ is referred to as a \textit{characteristic
triple} of $T$. The \textit{characteristic function} of $T$
corresponding to the triple $(\cle, B, D)$ is the $\clb(\cle,
\cld_{m, T^*})$-valued analytic function
\[
\Phi_T : \mathbb{B}^n \raro \clb(\cle, \cld_{m, T^*}),
\]
defined by
\[
\Phi_T(z) = \sum_{\bk \in \Zp} \sqrt{\rho_{m-1}(\bk)} D_{\bk}
z^{\bk} + \dmt (I_{\clh} - ZT^*)^{-m}ZB  \quad  \quad \big(\z \in
\mathbb{B}^n \big),
\]
where $Z = (z_1 I_{\clh}, \ldots, z_n I_{\clh})$ for all $\z \in \mathbb{B}^n$, $D_{\bk}$, $\bk \in \Zp$, is the $\bk$-th entry of the
``column matrix'' $D$ (see \eqref{eq-column D} for more details).

The operator-valued analytic function $\Phi_T$ may be viewed as a
counterpart of Sz.-Nagy and Foias characteristic functions for
contractions. Indeed, in Theorem \ref{thm-analytic model} in Section \ref{sec-ch fn}, we prove that $\Phi_T$
defines a partially isometric multiplier from $\arv(\cle)$, the $\cle$-valued Drury-Arveson space over the open unit ball
$\mathbb{B}^n$ \cite{WA}, to $\mathbb{H}_m(\mathbb{B}^n, \cld_{m,T^*})$,
the $\cld_{m,T^*}$-valued weighted Bergman space over $\mathbb{B}^n$. Moreover,
\[
\mathbb{H}_m(\mathbb{B}^n, \cld_{m,T^*}) \ominus \Phi_T
\; \arv(\cle),
\]
is the model space of the pure $m$-hypercontraction $T$ in the sense of Muller and Vasilescu ~\cite{MV}.

Section \ref{sec-4} deals with universal multipliers corresponding
to weight sequences and parameterizations of wandering subspaces of commuting tuples of shift operators.
In Theorem \ref{factom} we prove that any multiplier
from a vector-valued Drury-Arveson space to a (class of) vector-valued
reproducing kernel Hilbert space on $\mathbb{B}^n$ can be factored
as a product of a universal multiplier (which depends only on the kernel
function and coefficient Hilbert space) and a Schur-Agler class of functions.
We also point out that the unique factorization property holds
in the setting of ``inner'' functions in several variables (see Theorem \ref{thm-unique factorization}).
Then, in Section \ref{sec -fact ch fn}, we turn to a canonical factorization of $\Phi_T$.
Recall that \cite{AT} given Hilbert spaces $\cle$ and $\clf$ and an analytic function $\Theta : \mathbb{B}^n \raro \clb(\cle, \clf)$,
$\Theta$ is a contractive multiplier from $\arv(\cle)$ to $\arv(\clf)$ if and only if there exist auxiliary Hilbert space $\clh$
and a unitary
\[
W: \clh \oplus \cle \raro \Big(\bigoplus_{i=1}^n \clh \Big) \oplus \clf,
\]
such that, writing $W$ as
\[
W = \begin{bmatrix}
A & B\\ C & D
\end{bmatrix},
\]
one has the following \textit{transfer function} realization (cf. \cite{AT})
\[
\Theta(\z) = D + C (I_{\clh} - Z A)^{-1} Z B \quad \quad (\z \in \mathbb{B}^n).
\]
In Theorem \ref{thm-factchar}, we prove that $\Phi_T$ factors through a (canonical) transfer function. More specifically
\[
\Phi_T(\z) = \Psi_{\beta(m), \cld_{m, T^*}}(\z) \tilde{\Phi}_T (\z),
\]
where
\[
\tilde{\Phi}_T(\z) = D + \cmt (I_{\clh} - ZT^*)^{-1} ZB,
\]
is the transfer function of the unitary matrix $U$ corresponding to $(\cle, B,D)$, and
\[
\Psi_{\beta(m), \cld_{m, T^*}}(\z) = \begin{cases} \begin{bmatrix} \cdots &
\sqrt{\rho_{m-1}(\bk)} z^{\bk}I_{\cld_{m,T^*}} & \cdots \end{bmatrix}_{\bk
\in \Zp} & \mbox{if~} m \geq 2
\\
\begin{bmatrix}I_{\cld_{1, T^*}} & 0 & 0 & \cdots \end{bmatrix} & \mbox{if~} m=1, \end{cases}
\]
for all $\z \in \mathbb{B}^n$. Here $\Psi_{\beta(m), \cld_{m, T^*}}$ is the universal multiplier from $\arv(\lnt)$ to $\bH_m(\mathbb{B}^n, \cld_{m, T^*})$. In the final section, Section \ref{sec 6}, we link up our results with characteristic functions of pure row contractions \cite{BES}.

%The paper is organized as follows: In the first part of Section \ref{sec-prel}, we introduce the central notions used in the paper. Then we introduce characteristic triples and natural unitary colligation matrices corresponding to pure $m$-hypercontractions. The characteristic function corresponding to a characteristic triple of a pure $m$-hypercontraction is then build in Section \ref{sec-ch fn}. Here we prove that the characteristic function $\Phi_T$ of a pure $m$-hypercontraction $T$ is a complete unitary invariants. Also we prove that the lattice of closed joint invariant subspaces of $T$ is determined by factors of $\Phi_T$. Section \ref{sec-4} deals with universal multipliers corresponding to weight sequences and parameterizations of wandering subspaces of commuting tuples of shift operators. In Section \ref{sec -fact ch fn}, we present a canonical factorization of characteristic functions of pure hypercontractions. We prove that the characteristic function of a pure hypercontraction can be represented as a canonical product of a universal multiplier and a Drury-Arveson multiplier.

\newsection{Preliminaries and Characteristic triples}\label{sec-prel}

We begin by exploring natural examples of pure $m$-hypercontractions. Let
$p$ be a natural number, and let
\[
K_p(\z, \w) = (1 - \displaystyle \sum_{i=1}^n z_i \bar{w}_i)^{-p}
\quad \quad (\z, \w \in \mathbb{B}^n).
\]
Then $K_p$ is a positive-definite kernel on $\mathbb{B}^n$. Denote
by $\bH_{p}$ the reproducing kernel Hilbert space (of scalar-valued
analytic functions on $\mathbb{B}^n$) corresponding to the kernel
$K_p$. If $\w \in \mathbb{B}^n$, then we let $K_p(\cdot, \w)$ denote
the function in $\bH_p$ defined by
\[
(K_p(\cdot, \w))(\z) = K_p(\z, \w) \quad \quad (\z \in \mathbb{B}^n).
\]
Given a Hilbert space $\cle$, we denote by $\bH_p(\mathbb{B}^n,
\cle)$ the reproducing kernel Hilbert space corresponding to the
$\clb(\cle)$-valued kernel
\[
(\z, \w) \mapsto K_p(\z, \w) I_{\cle},
\]
on $\mathbb{B}^n$. We simply write $\bH_{p}$ instead of
$\bH_{p}(\mathbb{B}^n, \mathbb{C})$ if $\cle =\mathbb{C}$. Note that
for $\z \in\mathbb{B}^n$, we have (cf. page 983, \cite{MV})
\[
(1 - \sum_{i=1}^n z_i)^{-p} = \sum_{\bk \in \Zp} \rho_p(\bk)
z^{\bk},
\]
where $z^{\bk} = z_1^{k_1} \cdots z_n^{k_n}$ for all $\bk \in \Zp$.
Then
\[
\bH_{p}(\mathbb{B}^n, \cle) = \{f = \sum_{\bk \in \Zp} a_{\bk}
z^{\bk} \in \clo(\mathbb{B}^n, \cle): \|f\|^2 :=\sum_{\bk \in \Zp}
\frac{\|a_{\bk}\|_{\cle}^2}{\rho_p({\bk})} < \infty \}.
\]
In particular, $\bH_{1}(\mathbb{B}^n, \cle)$, $\bH_{n}(\mathbb{B}^n,
\cle)$ and $\bH_{n+1}(\mathbb{B}^n, \cle)$ represents the well-known
$\cle$-valued \textit{Drury-Arveson space}, the
Hardy space and the Bergman space over $\mathbb{B}^n$, respectively.
Moreover, for each $p > n+1$, $\bH_{p}(\mathbb{B}^n, \cle)$ is an
$\cle$-valued weighted Bergman space over $\mathbb{B}^n$ (cf. \cite{Zhu}).
Following standard notation, we denote the Drury-Arveson space $\bH_{1}(\mathbb{B}^n, \cle)$
by $H^2_n(\cle)$. Again, if $\cle = \mathbb{C}$, then we simply denote $H^2_n(\mathbb{C})$ by $H^2_n$.

\NI An easy computation shows that $(M_{z_1}, \ldots, M_{z_n})$, the
commuting tuple of multiplication operators by the coordinate
functions $\{z_1, \ldots, z_n\}$, defines a pure $m$-hypercontraction on
$\bH_{p}(\mathbb{B}^n, \cle)$ for all $m \leq p$.

\NI To simplify the notation, we often identify $\bH_{p} \otimes
\cle$ with $\bH_{p}(\mathbb{B}^n, \cle)$ via the unitary map defined by
$\z^{\bk} \otimes \eta \mapsto \z^{\bk} \eta$, for all $\bk \in \Zp$ and $\eta \in \cle$.
As a consequence, we can identify $M_z \otimes I_{\cle}$ on
$\bH_{p} \otimes \cle$ with $M_z$ on $\bH_{p}(\mathbb{B}^n, \cle)$
(as tuples of operators).

Recall that a holomorphic map $\Phi : \mathbb{B}^n \raro
\clb(\cle_1, \cle_2)$, for some Hilbert spaces $\cle_1$ and
$\cle_2$, is said to be a \textit{multiplier} from $\arv(\cle_1)$
to $\bH_m(\mathbb{B}^n, \cle_2)$ if
\[
\Phi f \in \bH_m(\mathbb{B}^n, \cle_2),
\]
for all $f \in \arv(\cle_1)$. We denote by $\clm(\arv(\cle_1), \bH_m(\mathbb{B}^n, \cle_2))$ the set of all multipliers from
$\arv(\cle_1)$ to $\bH_m(\mathbb{B}^n, \cle_2)$. Note also
that a multiplier $\Phi \in \clm(\arv(\cle_1),
\bH_m(\mathbb{B}^n, \cle_2))$ gives rise to a bounded linear
operator
\[
M_{\Phi}: \arv(\cle_1) \raro \bH_m(\mathbb{B}^n, \cle_2), \quad
\quad f \mapsto \Phi f,
\]
known as the multiplication operator corresponding to $\Phi$.
Multipliers can be characterized as follows: Let $X \in
\clb(\arv(\cle_1), \bH_m(\mathbb{B}^n, \cle_2))$. Then $X \in
\clm(\arv(\cle_1), \bH_m(\mathbb{B}^n, \cle_2))$ if and only if
\[
X ( M_{z_i}\otimes I_{\cle_1}) = ( M_{z_i}\otimes I_{\cle_2}) X,
\]
for all $i = 1, \ldots, n$. For more details about multipliers on reproducing
kernel Hilbert spaces in our present context, we refer to \cite{JS2}.

Finally, recall also that if $T$ is a pure
$m$-hypercontraction on $\clh$, then the \textit{canonical dilation
map} (see \cite{MB}, and also see \cite{MV}) $\Pi_m : \clh \to
\mathbb{H}_m(\mathbb{B}^n, \cld_{m,T^*})$, defined by
\begin{equation}\label{eq-Pim}
(\Pi_mh)(\z) = D_{m,T^*}(I_{\clh} - ZT^*)^{-m} h \quad \quad
(h\in\clh, \z \in \mathbb{B}^n),
\end{equation}
is an isometry and
\[
\Pi_mT_i^* = M_{z_i}^* \Pi_m,
\]
for all $i = 1, \ldots, n$, where $Z : \clh^n \raro \clh$ is the row
contraction $Z = (z_1 I_{\clh}, \ldots, z_n I_{\clh})$, $\z
\in \mathbb{B}^n$. In particular, if
\[
\clq_{m,T} = \Pi_m \clh,
\]
then $\clq_{m,T}$, the \textit{canonical model space} corresponding
to $T$, is a joint $(M_{z_1}^*, \ldots, M_{z_n}^*)$-invariant
subspace and $(P_{\clq_{m,T}} M_{z_1}|_{\clq_{m,T}}, \ldots,
P_{\clq_{m,T}} M_{z_1}|_{\clq_{m,T}})$ on $\clq_{m,T}$ and $(T_1,
\ldots, T_n)$ on $\clh$ are unitarily equivalent (see \cite{MV,
JS2}). This shows, in particular, that pure $m$-hypercontractions on
Hilbert spaces are precisely (in the sense of unitary equivalence)
the compressions of $M_z$ to joint co-invariant subspaces of
vector-valued $\bH_m$-spaces.

\NI On the other hand, $\cls_{m,T}$, the \textit{canonical invariant
subspace} corresponding to $T$, defined by
\[
\cls_{m,T} = \mathbb{H}_m(\mathbb{B}^n, \cld_{m,T^*}) \ominus \Pi_m
\clh,
\]
is a joint $(M_{z_1}, \ldots, M_{z_n})$-invariant subspace of
$\bH_m(\mathbb{B}^n, \cld_{m,T^*})$, and hence by a
Beurling-Lax-Halmos type theorem (see \cite[Theorem 4.4]{JS2}) it
follows that
\begin{equation}\label{eq-SmT}
\cls_{m,T} =\Phi \arv(\cle),
\end{equation}
for some Hilbert space $\cle$ and a partially isometric multiplier
$\Phi \in \clm(\arv(\cle), \mathbb{H}_m(\mathbb{B}^n,
\cld_{m,T^*}))$.

We turn now to the main content of this section. Let $T$ be a pure
$m$-hypercontraction on $\clh$. Since
\[
\langle \z^{\bk}, \z^{\bl} \rangle_{\bH_m} = \frac{1}{\rho_m(\bk)}
\delta_{\bk, \bl},
\]
for all $\bk, \bl \in \Zp$, and the canonical dilation map $\Pi_m$
is an isometry, that is, $\Pi_m^* \Pi_m = I_{\clh}$, and
\[
(\Pi_m h)(\z) = \sum_{\bk \in \Zp} \rho_m(\bk)(D_{m,T^*}
T^{*\bk}h) \z^{\bk} \quad \quad (\z \in \mathbb{B}^n, h \in \clh),
\]
it follows that
\begin{equation}\label{eq-gamma I}
\sum_{\bk \in \Zp} \rho_m(\bk) T^{\bk} D^2_{m,T^*} T^{*\bk}
=I_{\clh}.
\end{equation}
Moreover, since
\[
\rho_{0}(\bk) =
\begin{cases} 1 & \mbox{if}~ \bk = \bm{0}
\\
0 & \mbox{otherwise}, \end{cases}
\]
an easy computation shows that (cf. page 96, \cite{ES})
\begin{equation}\label{eq-gamma m-1}
\rho_{m}(\bk) = \rho_{m-1}(\bk) + \mathop{\sum_{i=1}^n}_{k_i \geq 1}
\rho_m(\bk - \bm{e}_i),
\end{equation}
where
\[
\bk - \bm{e}_i =
\begin{cases} (k_1, \ldots, k_{i-1}, k_i - 1, k_{i+1},
\ldots, k_n) & \mbox{if}~ k_i \geq 1
\\
0 & \mbox{if}~ k_i = 0, \end{cases}
\]
and $\bk \in \Zp$. Hence, by \eqref{eq-gamma I} we have
\[
\begin{split}
\sum_{\bk \in \Zp} \rho_{m-1}(\bk) T^{\bk} D^2_{m,T^*} T^{*\bk}
\leq \sum_{\bk \in \Zp} \rho_m(\bk) T^{\bk} D^2_{m,T^*} T^{*\bk}
=I_{\clh}.
\end{split}
\]
Then the linear map $\cmt: \clh \to \lnt$ defined by
\begin{equation}\label{eq-cmt}
\cmt(h) = (\sqrt{\rho_{m-1}(\bk)} \dmt T^{*\bk} h)_{\bk \in \Zp}
\quad \quad \quad (h\in\clh),
\end{equation}
is a contraction. It is often convenient to represent $\cmt$ as the
column matrix
\begin{equation}\label{eq-cmt matrix}
\cmt = \begin{bmatrix} \vdots \\ \sqrt{\rho_{m-1}(\bk)}
\dmt T^{*\bk}
\\ \vdots
\end{bmatrix}_{\bk \in \Zp}.
\end{equation}
Now using \eqref{eq-gamma I} twice, we have
\[
\begin{split}
I_{\clh} - \cmt^* \cmt & = \sum_{\bk \in \Zp} \rho_m(\bk)
T^{\bk}D^2_{m,T^*}T^{*\bk} - \sum_{\bk \in \Zp} \rho_{m-1}(\bk)
T^{\bk}D^2_{m,T^*}T^{*\bk}
\\
& = \sum_{\bk \in \Zp} \Big( \rho_m(\bk) - \rho_{m-1}(\bk)\Big)
T^{\bk}D^2_{m,T^*}T^{*\bk}
\\
&= \sum_{\bk \in \Zp} \Big( \mathop{\sum_{i=1}^n}_{k_i \geq 1}
\rho_m(\bk - \bm{e}_i) \Big)T^{\bk}D^2_{m,T^*}T^{*\bk}
\\
&= \sum_{i=1}^n \sum_{\bk \in \Zp} \rho_m(\bk)T^{\bk+ \bm{e}_i}
D^2_{m,T^*}T^{* (\bk + \bm{e}_i)}
\\
&= \sum_{i=1}^n T_i \Big( \sum_{\bk \in \Zp}
\rho_m(\bk)T^{\bk}D^2_{m,T^*}T^{*\bk} \Big) T^*_i
\\
& = \sum_{i=1}^n T_i T_i^*,
\end{split}
\]
that is
\begin{equation}\label{eq-TCI}
TT^* + \cmt^* \cmt  = I_{\clh},
\end{equation}
and therefore
\[
X_{T} = \Bigl[\begin{matrix} T^* \\ \cmt
\end{matrix} \Bigr]: \clh \to \clh^n \oplus \lnt,
\]
is an isometry. By adding a suitable Hilbert space $\cle$, we extend
$X_{T}$ on $\clh$ to a unitary $U : \clh \oplus \cle \to \clh^n
\oplus \lnt$, and set
\[
U: = \begin{bmatrix} X_{T} & Y_T \end{bmatrix} : \clh \oplus \cle \to \clh^n \oplus \lnt,
\]
where $Y_T =  U|_{\cle}: \cle \raro \clh^n \oplus \lnt$. If we set $Y_T = \begin{bmatrix} B \\ D \end{bmatrix}$
where $B  = P_{\clh^n} Y_T \in \clb(\cle, \clh^n)$ and $D = P_{\lnt}
Y_T \in \clb (\cle, \lnt)$, then
\[
U = \begin{bmatrix}
T^* & B \cr \cmt  & D \cr
\end{bmatrix} : \clh \oplus \cle \to \clh^n \oplus \lnt.
\]
Summarizing, we have the following result.

\begin{Theorem}\label{thm-ch triple}
Let $T$ be a pure $m$-hypercontraction on $\clh$.
Then the map $\cmt: \clh \to \lnt$ defined by
\[
\cmt(h) = \Big(\sqrt{\rho_{m-1}(\bk)} \dmt T^{*\bk} h\Big)_{\bk \in
\Zp} \quad \quad \quad (h\in\clh),
\]
 is a contraction, and there exist a Hilbert space $\cle$ and a
bounded linear operator
\[
Y_T = \begin{bmatrix} B \\ D \end{bmatrix}: \cle \raro \clh^n \oplus \lnt,
\]
such that
\[
\begin{bmatrix} X_T & Y_T \end{bmatrix}   =
\begin{bmatrix}
T^* & B \cr \cmt  & D \cr
\end{bmatrix} : \clh \oplus \cle \to
\clh^n \oplus \lnt,
\]
is a unitary.
\end{Theorem}

This motivates the following definition: Let $T$ be a pure
$m$-hypercontraction on $\clh$, $m \geq 1$. A triple $(\cle, B, D)$
consisting of a Hilbert space $\cle$ and bounded linear operators $B
\in \clb(\cle, \clh^n)$ and $D \in \clb (\cle, \lnt)$ is said to be
a \textit{characteristic triple} of $T$ if
\[
\begin{bmatrix} X_T & Y_T \end{bmatrix} :=
\begin{bmatrix}
T^* & B \cr \cmt  & D \cr
\end{bmatrix} : \clh \oplus \cle \to
\clh^n \oplus \lnt
\]
is a unitary.

Characteristic triple of a pure $m$-hypercontraction is unique in the following sense:

\begin{Theorem}\label{Unique:triple}
Let $T$ be a pure $m$-hypercontraction on $\clh$, and let $(\cle, B, D)$ and $(\tilde{\cle},\tilde{B}, \tilde{D})$ be characteristic triples of $T$. Then there exists a unitary $U : \tilde{\cle}\to \cle$ such that
\[
(\tilde{\cle},\tilde{B}, \tilde{D})=(U^*\cle, BU, DU).
\]
\end{Theorem}
\begin{proof}
Since
\[
 \begin{bmatrix} X_{T} & Y_T \end{bmatrix}=
 \begin{bmatrix} T^*& B \cr \cmt & D\cr  \end{bmatrix}: {\clh} \oplus \cle \to
\clh^n \oplus l^2(\Zp, \cld_{m, T^*}),
\]
and
\[
 \begin{bmatrix} X_{T} & \tilde{Y}_T \end{bmatrix}=
 \begin{bmatrix} T^*& \tilde{B} \cr \cmt & \tilde{D}\cr  \end{bmatrix}: {\clh} \oplus \tilde{\cle} \to
\clh^n \oplus l^2(\Zp, \cld_{m, T^*}),
\]
are unitary operators, it follows that $Y_T= \begin{bmatrix} B \\ D \end{bmatrix}$ and $\tilde{Y}_T =  \begin{bmatrix} \tilde{B} \\ \tilde{D} \end{bmatrix}$
 are isometries and
\[
\mbox{ran} Y_T=\mbox{ran} \tilde{Y}_T.
\]
By Douglas lemma, we have
\[
\tilde{Y}_T = Y_T U,
\]
for some unitary $U : \tilde{\cle} \raro \cle$, and hence
\[
 \tilde{B} = B U \ \mbox{and}  \ \tilde{D} = D U.
 \]
This completes the proof.
\end{proof}

Characteristic triples of pure $m$-hypercontractions, $m \geq 1$, will play a key role in what follows. The special case $m=1$ will be treated in the
final section of this paper.

We conclude this section with an explicit construction of a characteristic triple of an $m$-hypercontraction $T$ on a Hilbert space $\clh$: Let $X_T$ be as above. Consider
\[
 \cle_T = \Big(\mbox{ran} X_{T}\Big)^\perp,
\]
and the inclusion map
\[
 i= \begin{bmatrix}
B_T \\ D_T
\end{bmatrix}: \Big(\mbox{ran} X_{T}\Big)^\perp \hookrightarrow \clh^n \oplus \lnt.
\]
Then it readily follows that $(\cle_T, B_T, D_T)$ is a characteristic triple of $T$.

\newsection{Characteristic Functions}\label{sec-ch fn}

In this section, we continue, by means of operator-valued analytic
functions corresponding to characteristic triples, the exploration
of pure $m$-hypercontractions. Here the operator-valued analytic
functions will play a similar role as Sz.-Nagy and Foias
characteristic functions for contractions.

Let $T$ be a pure $m$-hypercontraction on $\clh$, and let
$(\cle,B,D)$ be a characteristic triple of $T$. Note that $D$ can be represented by a column matrix
\begin{equation}\label{eq-column D}
D= \begin{bmatrix}
\vdots \\ D_{\bk} \\ \vdots
\end{bmatrix}_{\bk \in \Zp} : \cle \raro \lnt,
\end{equation}
where $D_{\bk} \in \clb(\cle, \cld_{m, T^*})$, $\bk \in \Zp$. Define
\[
\Phi_T: \mathbb{B}^n \to \clb(\cle, \cld_{m,T^*}),
\]
by
\begin{equation} \label{char fn}
\Phi_T(\z) = \Big(\sum_{\bk \in \Zp} \sqrt{\rho_{m-1}(\bk)} D_{\bk}
z^{\bk} \Big) + \dmt (I_{\clh} -ZT^*)^{-m}ZB  \quad  \quad (\z\in
\mathbb{B}^n).
\end{equation}
Notice that $\Phi_T$ is a $\clb(\cle, \cld_{m,T^*})$-valued analytic
function on $\mathbb{B}^n$. We call $\Phi_T$ the
\textit{characteristic function} of $T$ corresponding to the characteristic triple $(\cle,B,D)$.

We claim that $\Phi_T$ is a partially isometric multiplier from
$\arv(\cle)$ to $\bH_m(\mathbb{B}^n, \cld_{m, T^*})$. To this end, first
we proceed to compute $\Phi_T(\z) \Phi_T(\w)^*$, $\z, \w \in
\mathbb{B}^n$. For simplicity, set
\[
x_{\bk} = \sqrt{\rho_{m-1}(\bk)},
\]
and
\[
X(\z) = \sum_{\bk \in \Zp} x_{\bk} D_{\bk} z^{\bk}, \quad \quad
Y(\z) = \dmt (I_{\clh} -ZT^*)^{-m}ZB,
\]
for all $\z \in \mathbb{B}^n$ and $\bk \in \Zp$. Notice that, if $m = 1$, then $x_{\bk} = 0$ for all $\bk \in \Zp\setminus\{0\}$ and $x_{\bm{0}} = 1$. Thus
\[
\Phi_T(\z) \Phi_T(\w)^* = X(\z) X(\w)^* + X(\z) Y(\w)^* + Y(\z)
X(\w)^* + Y(\z) Y(\w)^*,
\]
for all $\z, \w \in \mathbb{B}^n$. On the other hand, since $\begin{bmatrix} T^* & B \cr \cmt  & D \cr
\end{bmatrix}$ is a co-isometry (see Theorem \ref{thm-ch triple}), we have
\begin{equation}\label{eq-matrix}
\begin{bmatrix}
T^* T + B B^* & T^* \cmt^* + B D^* \cr \cmt T + D B^* & \cmt \cmt^*
+ D D^*
\end{bmatrix} = \begin{bmatrix}
I_{\clh^n} & 0 \cr 0 & I_{\lnt}
\end{bmatrix}.
\end{equation}
Let $\z, \w \in \mathbb{B}^n$. We note that
\[
\begin{split}
X(\z) X(\w)^* & = \sum_{\bk, \bl \in \Zp} x_{\bk} x_{\bl} D_{\bk}
D_{\bl}^* z^{\bk} \bar{w}^{\bl}
\\
& = \sum_{\bk \in \Zp} x_{\bk}^2 D_{\bk} D_{\bk}^* z^{\bk}
\bar{w}^{\bk} + \sum_{\bk \neq \bl} x_{\bk} x_{\bl} D_{\bk}
D_{\bl}^* z^{\bk} \bar{w}^{\bl}.
\end{split}
\]
By \eqref{eq-matrix}, we have $\cmt \cmt^* + DD^* = I_{\lnt}$, which
implies
\[
x_{\bk}^2 \dmt T^{*\bk}T^{\bk}\dmt + D_{\bk} D^*_{\bk} =
I_{\cld_{m,T^*}},
\]
for all $\bk \in \Zp$, and
\[
x_{\bk} x_{\bl} \dmt T^{*\bk}T^{\bl}\dmt + D_{\bk} D^*_{\bl} =0,
\]
for all $\bk \neq \bl$ in $\Zp$. This implies that
\[
\sum_{\bk \in \Zp} x_{\bk}^2 D_{\bk} D_{\bk}^* z^{\bk} \bar{w}^{\bk}
= \sum_{\bk \in \Zp} x_{\bk}^2 (I_{\cld_{m,T^*}} - x_{\bk}^2 \dmt
T^{*\bk}T^{\bk}\dmt) z^{\bk} \bar{w}^{\bk},
\]
and
\[
\sum_{\bk \neq \bl} x_{\bk} x_{\bl} D_{\bk} D_{\bl}^* z^{\bk}
\bar{w}^{\bl} = - \sum_{\bk \neq \bl} x_{\bk}^2 x_{\bl}^2 \dmt
T^{*\bk}T^{\bl}\dmt z^{\bk} \bar{w}^{\bl}.
\]
Hence
\[
\begin{split}
X(\z) X(\w)^* & = \sum_{\bk \in \Zp} x_{\bk}^2 z^{\bk} \bar{w}^{\bk}
I_{\cld_{m,T^*}} - \sum_{\bk, \bl \in \Zp} x_{\bk}^2 x_{\bl}^2 \dmt
T^{*\bk}T^{\bl}\dmt z^{\bk} \bar{w}^{\bl}
\\
& = K_{m-1}(\z, \w) I_{\dmt} - \dmt \Big(\sum_{\bk \in \Zp}
x_{\bk}^2 z^{\bk} T^{*\bk}\Big) \Big( \sum_{\bl \in \Zp} x_{\bl}^2
\bar{w}^{\bl} T^{\bl} \Big)\dmt
\\
& = K_{m-1}(\z, \w) I_{\dmt} - \dmt (I - ZT^*)^{-(m-1)} (I - T
W^*)^{-(m-1)}\dmt.
\end{split}
\]
Here
\[
K_0(\z, \w) \equiv 1 \quad \quad (\z, \w \in \mathbb{B}^n).
\]
Now we compute
\[
\begin{split}
X(\z) Y(\w)^* & = \Big(\sum_{\bk \in \Zp} x_{\bk} D_{\bk} z^{\bk}\Big) \Big(B^* W^* (I- T W^*)^{-m}\dmt \Big)
\\
& = \sum_{\bk \in \Zp} x_{\bk} z^{\bk} \Big(D_{\bk} B^* \Big) W^* (I - T
W^*)^{-m}\dmt.
\end{split}
\]
By \eqref{eq-matrix}, we have $\cmt T + D B^* = 0$, that is
\[
x_{\bk} \dmt T^{*\bk} T + D_{\bk} B^* = 0 \quad \quad (\bk \in \Zp),
\]
and so
\[
X(\z) Y(\w)^* = - \dmt \Big(\sum_{\bk \in \Zp} x_{\bk}^2 z^{\bk}
T^{*\bk}\Big) T W^* (I- T W^*)^{-m}\dmt,
\]
that is
\[
X(\z) Y(\w)^* = - \dmt (I - ZT^*)^{-(m-1)} T W^* (1- T
W^*)^{-m}\dmt.
\]
By duality
\[
Y(\z) X(\w)^* = - \dmt (I- Z T^*)^{-m} Z T^* (I - T W^*)^{-(m-1)}
\dmt.
\]
Finally, again by \eqref{eq-matrix}, we have $T^* T + BB^* =
I_{\clh^n}$, and so
\[
\begin{split}
Y(\z) Y(\w)^* & = \dmt (I-ZT^*)^{-m} ZB B^* W^* (I - T W^*)^{-m}
\dmt
\\
& = \dmt (I-ZT^*)^{-m} Z(I_{\clh} - T^* T) W^* (I - T W^*)^{-m}
\dmt.
\end{split}
\]
Therefore
\[
\begin{split}
\Phi_T(\z) \Phi_T(\w)^* & = K_{m-1}(\z, \w) I_{\dmt} - \dmt (I -
ZT^*)^{-(m-1)} (I - T W^*)^{-(m-1)}\dmt
\\
& \quad \; - \dmt (I - ZT^*)^{-(m-1)} T W^* (I - T W^*)^{-m}\dmt
\\
& \quad \; - \dmt (I- Z T^*)^{-m} Z T^* (I - W T^*)^{-(m-1)} \dmt
\\
& \quad \; + \dmt (I - ZT^*)^{-m} Z(I - T^* T) W^* (I - T
W^*)^{-m} \dmt
\\
& = K_{m-1}(\z, \w) I_{\dmt} - \dmt (I - ZT^*)^{-m} M (I - T
W^*)^{-m}\dmt,
\end{split}
\]
where
\[
M = (I - ZT^*)(I - TW^*) + (I- ZT^*) TW^* + ZT^* (I - WT^*) - Z (I -
T^*T) W^*.
\]
This is now simplified to $M = I - Z W^*$, that is
\[
M = (1 - \langle \z, \w \rangle)I,
\]
and so
\[
\Phi_T(\z) \Phi_T(\w)^* = K_{m-1}(\z, \w) I_{\dmt} - (1 - \langle
\z, \w \rangle)\dmt (I - ZT^*)^{-m} (I - T W^*)^{-m}\dmt.
\]
We obtain
\begin{equation}\label{charid}
\begin{split}
\frac{1}{(1-\langle \z, \w \rangle)^m} I_{\cld_{m,T^*}} -
\frac{\Phi_T(\z) \Phi_T(\w)^*}{1-\langle \z, \w \rangle} =
\dmt(I - ZT^*)^{-m}(I - TW^*)^{-m}\dmt,
\end{split}
\end{equation}
which shows that
\[
(\z, \w) \in \mathbb{B}^n \times \mathbb{B}^n \mapsto \frac{1}{(1-\langle \z, \w
\rangle)^m}I_{\cld_{m,T^*}} - \frac{\Phi_T(\z)
\Phi_T(\w)^*}{1-\langle \z, \w \rangle},
\]
is a positive definite kernel. By a well-known fact from reproducing
kernel Hilbert space theory (cf. page 2412, \cite{BV2}), it follows
that
\[
\Phi_T \in \clm(\arv(\cle), \mathbb{H}_m(\mathbb{B}^n,
\cld_{m,T^*})),
\]
and hence
\[
M_{\Phi_T}^* \Big(K_m(\cdot, \w) \eta \Big) = K_1(\cdot, \w)
\Phi_T(\w)^* \eta \quad \quad (\w \in \mathbb{B}^m, \eta \in
\cld_{m, T^*}).
\]
This shows that
\[
(I - M_{\Phi_T} M_{\Phi_T}^*) \Big(K_m(\cdot, \w) \eta \Big)(\z) =
\Big(K_m(\z, \w) I_{\cld_{m,T^*}} - K_1(\z, \w) \Phi_T(\z)
\Phi_T(\w)^*\Big) \eta,
\]
and hence by \eqref{charid}
\[
(I - M_{\Phi_T} M_{\Phi_T}^*) \Big(K_m(\cdot, \w) \eta \Big)(\z) =
\dmt(I-ZT^*)^{-m}(I-TW^*)^{-m}\dmt \eta,
\]
for all $\z,\w \in \mathbb{B}^n$ and $\eta \in \cld_{m, T^*}$. On the
other hand, by the definition of canonical dilations (see
\eqref{eq-Pim}), $\Pi_m^* : \mathbb{H}_m(\mathbb{B}^n, \cld_{m,T^*})
\raro \clh$ is given by
\[
\Pi_m^* \Big( K_m(\cdot, \w) \eta \Big) = (I_{\clh} - T W^*)^{-m} \dmt
\eta \quad \quad (\w \in \mathbb{B}^m, \eta \in \cld_{m, T^*}).
\]
This implies that
\begin{equation}\label{eq-PimPim*}
\Pi_m \Pi_m^* \Big( K_m(\cdot, \w) \eta \Big) (\z)=
\dmt(I_{\clh} - ZT^*)^{-m}(I_{\clh} - TW^*)^{-m}\dmt \eta,
\end{equation}
for all $\z, \w \in \mathbb{B}^n$ and $\eta \in \cld_{m, T^*}$, and so
\[
\Pi_m \Pi_m^* = I_{\mathbb{H}_m(\mathbb{B}^n, \cld_{m, T^*})} -
M_{\Phi_T} M_{\Phi_T}^*.
\]
In particular, $M_{\Phi_T}$ is a partial isometry and the canonical
model invariant subspace corresponding to $T$ (see \eqref{eq-SmT})
is given by
\[
\cls_{m,T} = \Phi_T \arv(\cle).
\]
We have therefore proved the following:

\begin{Theorem}\label{thm-analytic model}
Let $T$ be a pure $m$-hypercontraction on $\clh$, and let $(\cle, B, D)$ be a characteristic triple of $T$. Then
\[
\Phi_T \in \clm(\arv(\cle),
\mathbb{H}_m(\mathbb{B}^n,\cld_{m,T^*})),
\]
is a partially isometric multiplier and
\[
\cls_{m, T} = \Phi_T \arv(\cle),
\]
where
\[
\Phi_T(z) = \sum_{\bk \in \Zp} \sqrt{\rho_{m-1}(\bk)} D_{\bk}
z^{\bk} + \dmt (I_{\clh} - ZT^*)^{-m}ZB  \quad  \quad \big(\z \in
\mathbb{B}^n \big),
\]
is the characteristic function corresponding to $(\cle, B, D)$ and $\cls_{m, T}$ is the canonical model invariant subspace corresponding
to $T$.
\end{Theorem}

Characteristic triples and functions are more explicit for $1$-hypercontractions (or row contractions). This particular case will be discussed in  Section \ref{sec 6}.

It is worth pointing out, also, that the representing multiplier
$\Phi_T$ of $\cls_{m, T}$ is unique up to a partial isometry constant right factor (cf. \cite[Theorem 6.5]{MB}):
If
\[
\cls_{m,T} = \tilde{\Phi} \arv(\tilde{\cle}),
\]
for some Hilbert space $\tilde{\cle}$ and partially isometric
multiplier $\tilde{\Phi} \in \clm(\arv(\tilde{\cle}),
\bH_m(\mathbb{B}^n, \cld_{m,T^*}))$, then there exists a partial
isometry $V \in \clb(\tilde{\cle}, \cle)$ such that
\[
\tilde{\Phi}(\z) = \Phi_T(\z) V \quad \quad (\z \in \mathbb{B}^n).
\]

We now proceed to prove complete unitary invariance of
characteristic triples of pure $m$-hypercontractions. Recall that
two commuting tuples $T=(T_1,\cdots,T_n)$ on $\clh$ and
$\tilde{T}=(\tilde{T}_1, \ldots, \tilde{T}_n)$ on $\tilde{\clh}$ are
said to be \textit{unitarily equivalent} if there exists a unitary $U
\in \clb(\clh, \tilde{\clh})$ such that $U T_i = \tilde{T}_i U$ for
all $i=1,\ldots,n$.

Let $T$ and $\Tt$ be pure $m$-hypercontractions on $\clh$ and
$\tilde{\clh}$, respectively. Let $\Phi_T$ and $\Phi_{\Tt}$ be characteristic
functions corresponding to characteristic triples $(\cle, B, D)$ and $(\tilde{\cle}, \tilde{B}, \tilde{D})$ of $T$ and $\Tt$,
respectively. The characteristic functions $\Phi_T$ and $\Phi_{\Tt}$
are said to \textit{coincide} if
\[
 \Phi_{\Tt}(\z) = \tau_* \Phi_T(\z) \tau \quad \quad (\z \in \mathbb{B}^n),
\]
for some unitary operators $\tau : \tilde{\cle} \raro \cle$ and
$\tau_* : \cld_{m,T^*} \to \cld_{m,\Tt^*}$. Characteristic triples
of pure $m$-hypercontractions are complete unitary invariants:

\begin{Theorem}\label{thm-cui}
Let $T$ and $\Tt$ be pure $m$-hypercontractions on $\clh$ and
$\tilde{\clh}$, respectively. Then $T$ and $\Tt$ are unitarily
equivalent if and only if characteristic functions of $T$ and $\Tt$ coincide.
% $\Phi_T$ and $\Phi_{\Tt}$ coincide.
\end{Theorem}

\NI \textsf{Proof.}
Let $\Phi_T$ and $\Phi_{\Tt}$ be characteristic
functions corresponding to characteristic triples $(\cle, B, D)$ and $(\tilde{\cle}, \tilde{B}, \tilde{D})$ of $T$ and $\Tt$,
respectively.
Then
\[
U = \begin{bmatrix} X_T & Y_T \end{bmatrix}
\in \clb(\clh \oplus \cle, \clh^n \oplus \lnt),
\]
and
\[
\tilde{U} =
\begin{bmatrix} X_{\Tt} & Y_{\Tt} \end{bmatrix} \in \clb(\tilde{\clh}
\oplus \tilde{\cle}, \tilde{\clh}^n \oplus l^2(\mathbb{Z}^n_+,
\cld_{m, \Tt^*}),
\] are unitaries corresponding to characteristic triples $(\cle,B,D)$ and $(\tilde{\cle}, \tilde{B}, \tilde{D})$,
respectively, as in Theorem \ref{thm-ch triple}.

To prove the forward implication, let $W : \clh \raro
\tilde{\clh}$ be a unitary such that $W T_i = \Tt_i W$,
$i=1,\ldots,n$. Then $W D_{m, T^*}  = D_{m, \Tt^*} W$, and so
\[
C_{m, \Tt} W = (I \otimes W|_{\cld_{m, T^*}}) C_{m, T}.
\]
Also we have unitaries
\[
 W_n := W \oplus \cdots \oplus W : \clh^n \raro \tilde{\clh}^n,
\]
and
\[
 \hat{W} := \begin{bmatrix}
W_n & 0 \\ 0 & I \otimes W|_{\cld_{m, T^*}}
\end{bmatrix} : \clh^n \oplus \lnt \raro \tilde{\clh}^n \oplus l^2(\mathbb{Z}^n_+, \cld_{m,
\Tt^*}),
\]
which gives
\[
W_n T^* = \Tt^* W \quad \mbox{and} \quad \hat{W} X_T = X_{\Tt} W.
\]
Hence
\[
\begin{bmatrix} X_{\Tt} & \hat{W} Y_T \end{bmatrix} = \hat{W} \begin{bmatrix} X_T & Y_T \end{bmatrix} \begin{bmatrix}W^* & 0 \\ 0 & I_{\cle}
\end{bmatrix}.
\]
In particular
\[
\begin{bmatrix} X_{\Tt} & \hat{W} Y_T \end{bmatrix}: \tilde{\clh} \oplus \cle \raro \tilde{\clh}^n \oplus l^2(\mathbb{Z}^n_+, \cld_{m, \Tt^*}),
\]
is a unitary and
\[
\hat{W} Y_T = \begin{bmatrix} W_n B \\ (I \otimes W) D \end{bmatrix}
\in \clb(\cle, \tilde{\clh}^n \oplus l^2(\mathbb{Z}^n_+, \cld_{m,
\Tt^*})),
\]
is an isometry. Thus, $(\cle, W_n B, (I \otimes W) D)$ is a
characteristic triple of $\Tt$ and hence, by Theorem ~\ref{Unique:triple}, there exists a unitary $V : {\cle} \raro
\tilde{\cle}$ such that $\hat{W} Y_T = Y_{\tilde{T}} V$. This shows
that
\[
Y_{\Tt} = \hat{W} Y_T V^*,
\]
that is
\[
Y_{\Tt} = \begin{bmatrix} W_n B V^* \\ (I \otimes W) D V^* \end{bmatrix},
\]
and so
\[
\tilde{U} = \begin{bmatrix} \Tt^* & W_n B V^* \\ C_{m, \Tt} & (I \otimes W) D V^* \end{bmatrix}.
\]
A routine computation then shows that
\[
\Phi_{\Tt}(\z)= W \Phi_T(\z) V^* \quad \quad (\z \in \mathbb{B}^n).
\]
In order to prove sufficiency, we let $\Phi_{\Tt}(\z) = \tau_*
\Phi_T(\z) \tau^*$ for all $\z \in \mathbb{B}^n$ for some unitaries
$\tau \in \clb(\cle, \tilde{\cle})$ and $\tau_* \in \clb(\cld_{m,
T^*}, \cld_{m, \Tt^*})$. Then
\[
M_{\Phi_{\Tt}} = (I_{\mathbb{H}_m} \otimes \tau_*) M_{\Phi_T}
(I_{\arv} \otimes \tau^*),
\]
and so
\[
(I_{\mathbb{H}_m} \otimes \tau_*^*)(I_{\mathbb{H}_m(\mathbb{B}^n,
\cld_{m,{\Tt}^*})} - M_{\Phi_{\Tt}} M_{\Phi_{\Tt}}^*) =
(I_{\mathbb{H}_m(\mathbb{B}^n, \cld_{m,{T}^*})} - M_{\Phi_T}
M_{\Phi_T}^*) (I_{\bH_m} \otimes \tau^*_*),
\]
that is
\[
(I_{\mathbb{H}_m} \otimes \tau_*^*) P_{\clq_{\Tt}} = P_{\clq_{T}}
(I_{\bH_m} \otimes \tau^*_*).
\]
It follows that
\[
 (I_{\bH_m} \otimes \tau^*_*) \clq_{\Tt}=\clq_T.
\]
Moreover
\[
\begin{split}
(I_{\mathbb{H}_m} \otimes \tau_*^*) \Big( P_{\clq_{\Tt}} M_{z_i} P_{\clq_{\Tt}} \Big) & = (I_{\mathbb{H}_m} \otimes \tau_*^*) P_{\clq_{\Tt}} M_{z_i}P_{\clq_{\Tt}}
\\
& = P_{\clq_{T}} (I_{\bH_m} \otimes \tau^*_*) M_{z_i}P_{\clq_{\Tt}}
\\
& = P_{\clq_{T}} M_{z_i} (I_{\bH_m} \otimes \tau^*_*)P_{\clq_{\Tt}},
\end{split}
\]
that is
\[
(I_{\mathbb{H}_m} \otimes \tau_*^*) \Big( P_{\clq_{\Tt}} M_{z_i} P_{\clq_{\Tt}} \Big)= \Big(P_{\clq_{T}} M_{z_i} P_{\clq_{T}} \Big) (I_{\bH_m} \otimes \tau^*_*),
\]
for all $i = 1, \ldots, n$. Combining with the previous equality, we conclude that
\[
P_{\clq_T} M_z|_{\clq_T} \cong P_{\clq_{\Tt}}M_z|_{\clq_{\Tt}},
\]
that is, $T\cong \Tt$. \qed

We now proceed to study joint invariant subspaces
of pure $m$-hypercontractions. Following Sz.-Nagy-Foias
factorizations of characteristic functions, we relate joint
invariant subspaces of pure $m$-hypercontractions with
operator-valued factors of characteristic functions corresponding to
characteristic triples. We make good use of the following fact (see Lemma 2, \cite{AT}):

\begin{Lemma}\label{AgMc}
Let $\cle$, $\cle_*$ and $\clf$ be Hilbert spaces, and let $\Phi$
and $\Psi$ be $\clb(\cle, \cle_*)$ and $\clb(\clf, \cle_*)$ valued
analytic functions, respectively, on $\mathbb{B}^n$. Then the
following are equivalent:
\begin{enumerate}
\item[\textup{(i)}]
$(\z, \w) \mapsto \frac{\Psi(\z)\Psi(\w)^* - \Phi(\z)\Phi(\w)^*}{1 - \langle \z, \w \rangle}$
is a positive-definite kernel on $\mathbb{B}^n$.
\item[\textup{(ii)}]
There exists a contractive multiplier $\Theta \in
\clm(\arv(\cle), \arv(\clf))$ such that
\[
\Phi(z) = \Psi(\z) \Theta(\z) \quad \quad (\z \in \mathbb{B}^n).
\]
\end{enumerate}
\end{Lemma}

We are now ready for a factorization theorem for joint invariant
subspaces of pure $m$-hypercontractions.

\begin{Theorem}\label{thm-joint inv sub}
Let $T$ be a pure $m$-hypercontraction on $\clh$, and let
$(\cle,B,D)$ be a characteristic triple of $T$. If $\Phi_T$ is the
characteristic function corresponding to $(\cle,B,D)$, then $T$ has
a closed joint invariant subspace if and only if there exist a
Hilbert space $\clf$, a contractive multiplier $\Phi_1 \in
\clm(\arv(\cle), \arv(\clf))$, and a partially isometric
multiplier $\Phi_2 \in \clm(\arv(\clf),
\mathbb{H}_m(\mathbb{B}^n,\cld_{m,T^*}))$ such that
\[
\Phi_T(\z) = \Phi_2(\z) \Phi_1(\z) \quad \quad (\z \in \mathbb{B}^n).
\]
Moreover, the joint-invariant subspace is non-trivial if and only if
$\mbox{ran}M_{\Phi_2}$ is neither equal to $\mbox{ran}M_{\Phi_T}$
nor to $\mathbb{H}_m(\mathbb{B}^n, \cld_{m,T^*})$.
\end{Theorem}
\NI\textit{Proof.} Let $\clh_1$ be a closed joint $T$-invariant
subspace of $\clh$, and let $\clh_2 = \clh \ominus \clh_1$. Then
\[
\mathbb{H}_m(\mathbb{B}^n, \cld_{m,T^*}) \ominus \Pi_m \clh_2,
\]
is a joint $M_z$-invariant subspace of
$\mathbb{H}_m(\mathbb{B}^n,\cld_{m,T^*})$. By a Beurling-Lax-Halmos
type theorem for weighted Bergman spaces (see Theorem 4.4,
\cite{JS2}), there exist a Hilbert space $\clf$ and a partially
isometric multiplier $\Phi_2 \in
\clm(\arv(\clf),\mathbb{H}_m(\mathbb{B}^n,\cld_{m,T^*}))$ such
that
\begin{equation}
\label{h2}
\mathbb{H}_m(\mathbb{B}^n,\cld_{m,T^*}) \ominus \Pi_m \clh_2 =
\Phi_2 \arv(\clf).
\end{equation}
Since $\clq_T = \Pi_m \clh$ and $\Pi_m \clh =
\mathbb{H}_m(\mathbb{B}^n,\cld_{m,T^*}) \ominus
\Phi_T\arv(\cle)$, we conclude that
\begin{align}
\label{h1}
% \begin{split}
\Pi_m \clh_1 &= \Pi_m \clh \ominus \Pi_m \clh_2\nonumber
\\
& = \Big(\mathbb{H}_m(\mathbb{B}^n,\cld_{m,T^*}) \ominus \Phi_T
\arv(\cle)\Big) \ominus
\Big(\mathbb{H}_m(\mathbb{B}^n,\cld_{m,T^*}) \ominus \Phi_2
\arv(\clf)\Big)\nonumber
\\
& = \Phi_2 \arv(\clf) \ominus \Phi_T\arv(\cle),
%\end{split}
\end{align}

and hence
\[
(\z, \w) \in \mathbb{B}^n \times \mathbb{B}^n \mapsto
\frac{\Phi_2(\z)\Phi_2(\w)^* - \Phi_T(\z)\Phi_T(\w)^*}{1- \la \z,\w
\ra},
\]
is a kernel of the reproducing kernel Hilbert space $\Pi_m
\clh_1$. By Lemma \ref{AgMc}, there is a contractive multiplier
$\Phi_1 \in \clm(\arv(\cle), \arv(\clf))$ such that $\Phi_T(\z)
= \Phi_2(\z) \Phi_1(\z)$ for all $\z \in \mathbb{B}^n$.

\NI To prove the converse, let $\clf$ be a Hilbert space, $\Phi_1
\in \clm(\arv(\cle), \arv(\clf))$ be a contractive multiplier,
$\Phi_2 \in \clm(\arv(\clf),
\mathbb{H}_m(\mathbb{B}^n,\cld_{m,T^*}))$ be a partially isometric
multiplier, and let $\Phi_T(\z) = \Phi_2(\z) \Phi_1(\z)$ for all $\z
\in \mathbb{B}^n$. We have
\[
\mbox{ran} M_{\Phi_T} \subseteq \mbox{ran} M_{\Phi_2},
\]
and hence
\[
\clq \subseteq \clq_T,
\]
where
\[
\clq = \Big(\mbox{ran} M_{\Phi_2} \Big)^\perp,
\]
is a joint $M_z^*$-invariant subspace of
$\mathbb{H}_m(\mathbb{B}^n,\cld_{m,T^*})$. It now follows that
$\clh_1 = \clh \ominus \Pi_m^* \clq$ is a joint $T$-invariant
subspace of $\clh$.

\NI For the last part, note that, by \eqref{h2}, the invariant
subspace $\clh_1$ of $T$ is the full space if and only if
$\mbox{ran} M_{\Phi_2} = \mathbb{H}_m(\mathbb{B}^n, \cld_{m,T^*})$. On the
other hand, by ~\eqref{h1}, $\clh_1=0$ if and only if $\mbox{ran}
M_{\Phi_2} = \mbox{ran} M_{\Phi_T}$. This completes the proof of the
theorem.
\qed

\newsection{Universal multipliers and wandering subspaces}\label{sec-4}

In \cite{BV2}, Ball and Bolotnikov proved the following: Given a
vector-valued weighted shift space $H^2(\beta, \cle_*)$ (see the
definition below), there exists a \textit{universal multiplier}
$\psi_{\beta}$ (depending only on $\beta$ and $\cle_*$) such that any
contractive multiplier $\theta$ from a vector-valued Hardy space
$H^2_{\cle}(\D)$ to $H^2(\beta, \cle_*)$ factors through
$\psi_{\beta}$, that is
\[
\theta(z) = \psi_\beta(z) \tilde{\theta}(z) \quad \quad (z \in \D),
\]
for some Schur multiplier $\tilde{\theta} \in \clm(H^2_{\cle}(\D),
H^2_{l^2(\cle_*)}(\D))$ (see \cite[Theorem 2.1]{BV2} for more details).

In this section, we generalize the above to several
variables multipliers. We also define ``inner functions'' and
examine the uniqueness of universal factorizations in several
variables. First, we fix some notation and terminology.

A strictly decreasing sequence of positive numbers $\beta =
\{\beta_j\}_{j=0}^{\infty}$ is said to be a \textit{weight
sequence}, if $\beta_0 =1$ and
\begin{equation}\label{eq-beta}
\liminf \beta_j^{\frac{1}{j}} \geq 1.
\end{equation}
For a Hilbert space $\cle$ and a weight sequence $\beta$, we let
$\wbg(\beta, \cle)$ denote the Hilbert space of all $\cle$-valued
analytic functions $f = \displaystyle \sum_{\bk \in \Zp} a_{\bk}
z^{\bk}$, $a_{\bk} \in \cle$ for all $\bk \in \Zp$, on
$\mathbb{B}^n$ such that
\[
\|f\|^2_{\wbg(\beta, \cle)} := \displaystyle \sum_{j=0}^{\infty}
\beta_j \sum_{|\bk| = j} \frac{1}{\rho_1(\bk)} \|a_{\bk}
\|^2_{\cle}= \displaystyle \sum_{\bk \in \Zp}
\frac{\beta_{|\bk|}}{\rho_1(\bk)} \|a_{\bk} \|^2_{\cle} < \infty,
\]
that is
\[
\wbg(\beta, \cle) = \Big\{ f  \in \clo(\mathbb{B}^n, \cle):
\|f\|_{\wbg(\beta, \cle)} < \infty \Big\}.
\]
Then $\wbg(\beta, \cle)$ is an $\cle$-valued reproducing kernel
Hilbert space corresponding to the kernel
\begin{equation}\label{eq-kernel beta}
K_{\beta}(\z, \w) = \sum_{j=0}^{\infty}\frac{1}{\beta_j} \langle \z,
\w \rangle^j I_{\cle} \quad \quad (\z, \w \in \mathbb{B}^n).
\end{equation}
%In particular, if $\beta_j = 1$ for all $j \in \mathbb{Z}_+$, then
%$\wbg(\beta, \cle)$ is the Drury-Arveson space $\arv(\cle)$.
In particular, for $\beta_j = \frac{j! (n-1)!}{(n+j-1)!}$ and
$\beta_j = \frac{j! n!}{(n+j)!}$, $j \in \mathbb{Z}_+$,
$\wbg(\beta,\cle)$ represents the $\cle$-valued Hardy space and the
Bergman space over $\mathbb{B}^n$, respectively.

We now proceed to construct the universal multiplier corresponding
to the weight sequence $\beta$ and the Hilbert space $\cle$. Let
\[
\gamma_0 =1 \quad \mbox{ and } \gamma_j = \Big(\frac{1}{\beta_j} -
\frac{1}{\beta_{j-1}}\Big)^{-1} \quad \quad (j \geq 1).
\]
Then $\gamma = \{\gamma_j\}_{j \in \mathbb{Z}_+}$ is also a weight
sequence and hence
\[
K_{\gamma}(\z, \w) = \sum_{j=0}^{\infty} \frac{1}{\gamma_j} \langle
\z, \w \rangle^j \quad \quad(\z, \w \in \mathbb{B}^n),
\]
is a positive-definite kernel on $\mathbb{B}^n$. Define
$\Psi_{\beta, \cle}: \mathbb{B}^n \to \clb(l^2(\Zp, \cle), \cle)$ by
\[
\Psi_{\beta, \cle}(\z)(\{a_{\bk}\}_{\bk\in\Zp}) = \sum_{\bk \in \Zp}
\Big(\sqrt{\frac{\rho_1(\bk)}{\gamma_{|\bk|}}} a_{\bk}\Big) z^{\bk},
\]
for all $\z \in \mathbb{B}^n$ and $\{a_{\bk}\}_{\bk \in \Zp} \in
l^2(\Zp,\cle)$. We must first show that $\Psi_{\beta, \cle}$ is
well-defined. For each $\z \in \mathbb{B}^n$ and $\{a_{\bk}\}_{\bk
\in \Zp} \in l^2(\Zp,\cle)$, we have
\[
\begin{split}
\Big\| \sum_{\bk \in \Zp}
\Big(\sqrt{\frac{\rho_1(\bk)}{\gamma_{|\bk|}}} a_{\bk}\Big) z^{\bk}
\Big\|_{\cle} & \leq \sum_{\bk \in \Zp}
\sqrt{\frac{\rho_1(\bk)}{\gamma_{|\bk|}}} | z^{\bk}|
\|a_{\bk}\|_{\cle}
\\
& \leq \Big( \sum_{\bk \in \Zp} \frac{\rho_1(\bk)}{\gamma_{|\bk|}}
|z|^{2\bk} \Big)^{\frac{1}{2}} \Big( \sum_{\bk\in\Zp}
\|a_{\bk}\|_{\cle}^2 \Big)^{\frac{1}{2}}
\\
& = \Big( \sum_{j=0}^{\infty} \frac{1}{\gamma_j} \langle \z, \z
\rangle^j \Big)^{\frac{1}{2}}  \|\{a_{\bk}\}_{\bk \in
\Zp}\|_{l^2(\Zp, \cle)}
\\
& = K_{\gamma}(\z, \z)^{\frac{1}{2}} \|\{a_{\bk}\}_{\bk \in
\Zp}\|_{l^2(\Zp, \cle)},
\end{split}
\]
that is
\[
\|\Psi_{\beta, \cle}(\z)(\{a_{\bk}\}_{\bk\in\Zp})\| \leq
K_{\gamma}(\z, \z)^{\frac{1}{2}} \|\{a_{\bk}\}_{\bk \in
\Zp}\|_{l^2(\Zp, \cle)}.
\]
It is again convenient to represent $\Psi_{\beta}(\z)$,
$\z\in\mathbb{B}^n$, as the row operator
\[ \
\Psi_{\beta, \cle}(\z)= \Big[\cdots
\sqrt{\frac{\rho_1(\bk)}{\gamma_{|\bk|}}} z^{\bk}I_{\cle} \cdots
\Big]_{\bk\in \Zp}.
\]
Now we prove that:

\begin{Lemma}
$\Psi_{\beta, \cle} \in \clm(\arv\big(l^2(\Zp,\cle)\big),
\wbg(\beta, \cle))$ and
\[
M_{\Psi_{\beta, \cle}}: \arv\big(l^2(\Zp,\cle)\big)\to
\wbg(\beta, \cle),
\]
is a co-isometry.
\end{Lemma}
\NI \textsf{Proof:} For $\z$ and $\w$ in $\mathbb{B}^n$, we have
\[
\begin{split}
(1- \langle \z,\w \rangle)K_{\beta}(\z,\w) & =
\sum_{j=0}^{\infty}\frac{1}{\beta_j} \langle \z,\w \rangle^j -
\sum_{j=0}^{\infty}\frac{1}{\beta_j} \langle \z,\w \rangle^{j+1}
\\
& = \frac{1}{\beta_0} + \sum_{j=0}^{\infty} \Big(
\frac{1}{\beta_{j+1}}
\langle \z,\w \rangle^{j+1} - \frac{1}{\beta_j} \langle \z,\w \rangle^{j+1} \Big)
\\
& = \frac{1}{\beta_0} + \sum_{j=0}^{\infty}\Big(
\frac{1}{\beta_{j+1}} -
\frac{1}{\beta_j} \Big) \langle \z,\w \rangle^{j+1} \\
& = \sum_{j=0}^{\infty} \frac{1}{\gamma_j} \langle \z,\w \rangle^j,
\end{split}
\]
that is
\[
(1- \langle \z,\w \rangle)K_{\beta}(\z,\w) = K_{\gamma}(\z,\w).
\]
Hence from the matrix representation of $\Psi_{\beta, \cle}$ it follows that
\[
\begin{split}
\Psi_{\beta, \cle}(\z) \Psi_{\beta, \cle}(\w)^* & = \sum_{j=0}^{\infty}
\frac{1}{\gamma_j}
\sum_{|\bk|=j} \rho_{\bk} z^{\bk}\bar{w}^{\bk}I_{\cle} \\
& = \sum_{j=0}^{\infty} \frac{1}{\gamma_j} \langle \z,\w \rangle^j I_{\cle}\\
& = K_{\gamma}(\z,\w)I_{\cle} \\
& = (1- \langle \z,\w \rangle)K_{\beta}(\z,\w)I_{\cle},
\end{split}
\]
which implies
\begin{equation}\label{psibeta}
K_{\beta}(\z,\w)I_{\cle} - \frac{\Psi_{\beta}(\z)
\Psi_{\beta}(\w)^*}{1- \langle \z,\w \rangle} = 0,
\end{equation}
and so $\Psi_{\beta, \cle} \in \clm(\arv\big(l^2(\Zp,\cle)\big),
\wbg(\beta, \cle))$. The remaining part of the lemma follows from
\eqref{psibeta} and the fact that $\{K_{\beta}(\cdot, \w) \eta: \w
\in \mathbb{B}^n, \eta \in \cle\}$ is a total set in $\arv(\beta,
\cle)$. \qed

Given Hilbert spaces $\cle$ and $\cle_*$, we use $\cls
\clm(\arv(\cle), \wbg(\beta, \cle_*))$ to denote the set of all
contractive multipliers, that is
\[
\cls \clm(\arv(\cle), \wbg(\beta, \cle_*)) = \{\Phi \in
\clm(\arv(\cle), \wbg(\beta, \cle_*)): \|M_\Phi\| \leq 1\}.
\]
Now we are ready to prove the main theorem of this section.

\begin{Theorem} \label{factom}
Let $\cle$ and $\cle_*$ be Hilbert spaces, $\beta$ be a weight
sequence, and let $\Theta: \mathbb{B}^n \raro \clb(\cle, \cle_*)$ be
an analytic function. Then $\Theta \in \cls
\clm(\arv(\cle), \wbg(\beta, \cle_*))$ if and only if there
exists a multiplier $\tilde{\Theta} \in \cls \clm(\arv(\cle),
\arv(l^2(\Zp,\cle_*))$ such that
\[
\Theta(\z)= \Psi_{\beta, \cle_*}(\z) \tilde{\Theta}(\z) \quad \quad
(\z \in \mathbb{B}^n).
\]
\end{Theorem}

\NI \textsf{Proof:} Let $\tilde{\Theta} \in \cls \clm(\arv(\cle),
\arv(l^2(\Zp,\cle_*))$, and let $\Theta(\z)=\Psi_{\beta,
\cle_*}(\z)\tilde{\Theta}(\z)$ for all $\z \in \mathbb{B}^n$. Then
\[
\begin{split}
K_{\beta}(\z,\w)I_{\cle_*} - \frac{\Theta(\z)\Theta(\w)^*}{1-
\langle \z,\w \rangle} & = K_{\beta}(\z,\w)I_{\cle_*} -
\frac{\Psi_{\beta,
\cle_*}(\z)\tilde{\Theta}(\z)\tilde{\Theta}(\w)^*\Psi_{\beta,
\cle_*}(\w)^*}{1- \langle \z,\w \rangle}
\\
& = \frac{\Psi_{\beta, \cle_*}(\z)\Psi_{\beta, \cle_*}(\w)^*}{1 -
\langle \z,\w \rangle} - \frac{\Psi_{\beta,
\cle_*}(\z)\tilde{\Theta}(\z)\tilde{\Theta}(\w)^*\Psi_{\beta,
\cle_*}(\w)^*}{1- \langle \z,\w \rangle}
\\
& = \Psi_{\beta, \cle_*}(\z) \Big[ \frac{I_{l^2(\Zp,\cle_*)}-
\tilde{\Theta}(\z)\tilde{\Theta}(\w)^*}{1 - \langle \z,\w \rangle}
\Big] \Psi_{\beta, \cle_*}(\w)^*,
\end{split}
\]
for all $\z, \w \in \mathbb{B}^n$, where the last but one equality
follows from \eqref{psibeta}. Since $\tilde{\Theta}$ is a
contractive multiplier, it follows that
\[
(\z, \w) \mapsto K_{\beta}(\z,\w)I_{\cle_*} - \frac{\Theta(\z)
\Theta(\w)^*}{1- \langle \z,\w \rangle},
\]
is a positive definite kernel on $\mathbb{B}^n$, and so $\Theta \in
\cls \clm(\arv(\cle), \wbg(\beta, \cle_*))$. To prove the
converse we first note that $M_{\Theta}: \arv(\cle) \to
\wbg(\beta, \cle_*)$ is a contraction. Again, by \eqref{psibeta},
we have
\[
\begin{split}
K_{\beta}(\z,\w)I_{\cle_*} - \frac{\Theta(\z) \Theta(\w)^*}{1-
\langle \z,\w \rangle} & = \frac{\Psi_{\beta, \cle_*}(\z)
\Psi_{\beta, \cle_*}(\w)^*}{1- \langle \z,\w \rangle} -
\frac{\Theta(\z) \Theta(\w)^*}{1- \langle \z, \w \rangle}
\\
& = \frac{\Psi_{\beta, \cle_*}(\z)\Psi_{\beta, \cle_*}(\w)^* -
\Theta(\z) \Theta(\w)^*}{1- \langle \z,\w \rangle},
\end{split}
\]
for all $\z, \w \in \mathbb{B}^n$. Hence
\[
(\z, \w) \mapsto \frac{\Psi_{\beta, \cle_*}(\z)\Psi_{\beta,
\cle_*}(\w)^* - \Theta(\z) \Theta(\w)^*}{1- \langle \z,\w \rangle} \in\clb(\cle_*),
\]
is a positive-definite kernel on $\mathbb{B}^n$. The proof now
follows from Lemma \ref{AgMc}. \qed

The above theorem implies that the following diagram is commutative:

\setlength{\unitlength}{3mm}
 \begin{center}
 \begin{picture}(40,16)(0,0)
\put(12.7,3){$\arv(\cle)$}\put(19,1.6){$M_{\Theta}$}
\put(22.9,3){$\wbg(\beta, \cle_*)$} \put(22, 10){$\arv(l^2(\Zp,
\cle_*))$} \put(24,9.2){ \vector(0,-1){5}} \put(15.8,
4.3){\vector(1,1){5.5}} \put(16.4,
3.4){\vector(1,0){6}}\put(15.8,6.8){$M_{\tilde{\Theta}}$}\put(25.3,6.5){$M_{\Psi_{\beta, \cle_*}}$}
\end{picture}
\end{center}

We now turn to ``inner functions'' in $\clm(\arv(\cle),
\wbg(\beta, \cle_*))$. The concept of inner functions in the setting
of Bergman space (knows as the Bergman inner functions) is due to
Hedenmalm \cite{HH} (see also Olofsson \cite{AO1} and Eschmeier
\cite{ES} for weighted Bergman spaces in one and several variables,
respectively). The notion of inner functions (or $K$-inner
functions) in several variables was introduced in \cite{MB}.

A contractive multiplier $\Theta \in \clm(\arv(\cle),
\wbg(\beta, \cle_*))$ is said to be \textit{$K_{\beta}$-inner} if
\[
\|\Theta h\|_{\wbg(\beta,\cle_*)}= \|h\|_{\cle},
\]
for all $h \in \cle$ (that is, $M_{\Theta}|_{\cle}$ is an isometry), and
\[
\Theta \cle \perp  z^{\bk} \Theta \cle \quad \quad (\bk \in \Zp).
\]
In the case when  $\Theta \in \clm(\arv(\cle),
\arv( \cle_*))$ we simply say $\Theta$ is a $K$-inner multiplier.

In connection with this notice also that for a Hilbert space $\cle$
\[
 M_{\Psi_{\beta, \cle}}|_{l^2(\Zp,\cle)} : l^2(\Zp,\cle) \raro \wbg(\gamma,
 \cle),
\]
is an isometry. Indeed, for each
$\{a_{\bk}\}_{\bk \in \Zp} \in l^2(\Zp,\cle)$ and $\z\in
\mathbb{B}^n$, we have
\[
(M_{\Psi_{\beta, \cle}}(\{a_{\bk}\}_{\bk \in \Zp}))(\z) =
\Psi_{\beta, \cle}(\z)(\{a_{\bk}\}_{\bk\in\Zp}),
\]
and so
\[
\begin{split}
\|M_{\Psi_{\beta, \cle}}
(\{a_{\bk}\}_{\bk\in\Zp})\|^2_{\wbg(\gamma, \cle)} & =
\Big\|\sum_{\bk \in \Zp}
\Big(\sqrt{\frac{\rho_1(\bk)}{\gamma_{|\bk|}}} a_{\bk} \Big)
z^{\bk}\Big\|_{\wbg(\gamma, \cle)}^2
\\
& = \sum_{\bk \in \Zp} \frac{\rho_1(\bk)}{\gamma_{|\bk|}}
\|a_{\bk}\|^2_{\cle} \|z^{\bk}\|^2_{\wbg(\gamma, \cle)}
\\
& = \sum_{\bk \in \Zp} \frac{\rho_1(\bk)}{\gamma_{|\bk|}}
\|a_{\bk}\|^2_{\cle} \frac{\gamma_{|\bk|}}{\rho_1(\bk)}
\\
& = \|\{a_{\bk}\}_{\bk \in \Zp}\|_{l^2(\Zp,\cle)}^2.
\end{split}
\]

We now relate the idea of universal multipliers to uniqueness of
factorizations of multipliers in the context of Theorem
\ref{factom}.

\begin{Theorem}\label{thm-unique factorization}
Let $\cle$ and $\cle_*$ be Hilbert spaces, $\beta$ be a weight sequence, and let
\[
\Theta \in \cls \clm(\arv(\cle), \wbg(\beta, \cle_*)).
\]
If $\Theta$ is a $K_{\beta}$-inner multiplier then there exists a unique $K$-inner multiplier
\[
\tilde{\Theta} \in \cls \clm(\arv(\cle),
\arv(l^2(\Zp,\cle_*)),
\]
such that
\[
\Theta(\z)= \Psi_{\beta, \cle_*}(\z) \tilde{\Theta}(\z) \quad \quad
(\z \in \mathbb{B}^n).
\]
\end{Theorem}

\NI\textit{Proof.} If $\Theta \in \cls \clm(\arv(\cle),
\wbg(\beta, \cle_*))$, then by Theorem~\ref{factom}, we have
\[
\Theta = \Psi_{\beta, \cle_*} \tilde{\Theta}.
\]
for some $\tilde{\Theta} \in \cls \clm(\arv(\cle),
\wbg(l^2(\Zp,\cle_*))$. Now let ${\Theta}$ be
$K_{\beta}$-inner. We show that $\tilde{\Theta}$ is
$K$-inner. Let $\eta \in \cle$, and let
\[
\tilde{\Theta} \eta = f \oplus g \in \ker M_{\Psi_{\beta, \cle_*}}
\oplus (\ker M_{\Psi_{\beta, \cle_*}})^{\perp}.
\]
Since
\[
\|\eta\|_{\cle} = \|M_{\Theta} \eta\|_{\wbg(\beta, \cle_*)} =
\|M_{\Psi_{\beta, \cle_*}} M_{\tilde{\Theta}} \eta\|_{\wbg(\beta,
\cle_*)}
\]
and $M_{\Psi_{\beta, \cle_*}}$ is a co-isometry, we see that
\[
\begin{split}
\|\eta\|_{\cle} & = \|M_{\Psi_{\beta,\cle_*}} M_{\tilde{\Theta}}
\eta \|_{\wbg(\beta,\cle_*)}
\\
& = \|M_{\Psi_{\beta, \cle_*}} g \|_{\wbg(\beta,\cle_*)}
\\
& = \|g \|_{\arv (l^2(\Zp,\cle_*))}
\\
& \leq \|\tilde{\Theta} \eta \|_{\arv(l^2(\Zp,\cle_*))}
\\
& \leq \|\eta\|_{\cle} .
\end{split}
\]
It now follows that $\|\tilde{\Theta} \eta \|_{\arv (
l^2(\Zp,\cle_*) )} = \|\eta\|_{\cle}$ for all $\eta \in \cle$ and
\[
{\tilde{\Theta}} \cle \subseteq \Big(\ker M_{\Psi_{\beta, \cle_*}}
\Big)^\perp.
\]
This readily shows that
\[
M_{\Psi_{\beta, \cle_*}}^* M_{\Psi_{\beta, \cle_*}}|_{\tilde{\Theta} \cle} = I.
\]
Therefore, for $\eta, \zeta \in \cle$ and $\bk \in \mathbb{N}^n$, we see that
\[
\begin{split}
\langle \tilde{\Theta} \eta, z^{\bk} \tilde{\Theta} \zeta
\rangle_{\arv(l^2(\Zp, \cle_*))} & = \langle M_{\Psi_{\beta, \cle_*}}^* M_{\Psi_{\beta, \cle_*}} \tilde{\Theta} \eta,
\tilde{\Theta} z^{\bk} \zeta \rangle_{\arv(l^2(\Zp, \cle_*))}
\\
& = \langle \Psi_{\beta, \cle_*} \tilde{\Theta} \eta,
\Psi_{\beta, \cle_*} \tilde{\Theta} z^{\bk} \zeta
\rangle_{\wbg(\beta, \cle_*))}
\\
& = \langle {\Theta} \eta, {\Theta} z^{\bk} \zeta \rangle_{\wbg(\beta, \cle_*)}
\\
& = \langle {\Theta} \eta, z^{\bk} {\Theta} \zeta \rangle_{\wbg(\beta, \cle_*)},
\end{split}
\]
and hence the orthogonality condition of $K_\beta$-inner multiplier
$\Theta$ implies that of $\tilde{\Theta}$. Finally, since
\[
M_{\tilde{\Theta}} (z^{\bk} \eta) = z^{\bk} \tilde{\Theta} \eta =
z^{\bk} M_{\Psi_{\beta, \cle_*}}^* M_{\Psi_{\beta, \cle_*}}
\tilde{\Theta} \eta = z^{\bk} M_{\Psi_{\beta, \cle_*}}^* {\Theta}
\eta,
\]
for all $\eta \in \cle$ and $\bk \in \Zp$, it follows that $\tilde{\Theta}$ is unique.
%The sufficient part is trivial.
This completes the proof of the theorem. \qed

In the particular case $n=1$, all the results obtained so far in
this section are due to Ball and Bolotnikov \cite{BV2}.

The discussion to this point motivates us to define wandering
subspaces of bounded linear operators. The notion of a wandering subspace was introduced by Halmos
\cite{H} in the context of invariant subspaces of shifts on vector-valued Hardy spaces. Let $T$ be an $n$-tuple of commuting operators on $\clh$,
and let $\clw$ be a closed subspace of $\clh$. If
\[
\clw \perp T^{\bk} \clw,
\]
for all $\bk \in \mathbb{N}^n$, then $\clw$ is called a \textit{wandering subspace} for $T$. We say
that $\clw$ is a \textit{generating wandering subspace} for $T$ if
in addition
\[
\clh = \overline{\mbox{span}}\{T^{\bk} \clw : \bk \in \Zp\}.
\]
Here, however, we aim at parameterizing wandering subspaces for $M_z
= (M_{z_1}, \ldots, M_{z_n})$ on $\wbg(\beta, \cle_*)$. Note, by virtue of \eqref{eq-beta} and \eqref{eq-kernel beta}, that the tuple of multiplication operator $M_z$ defines a pure row contraction on $\wbg(\beta, \cle_*)$. Let $\clw$
be a wandering subspace for $M_z$ on $\wbg(\beta, \cle_*)$. Clearly
\[
\bigvee_{\bk\in\Zp} {z}^{\bk}\clw,
\]
is a joint $M_z$-invariant subspace of $\wbg(\beta, \cle_*)$. Then
there exist a Hilbert space $\cle$ and a partial isometric
multiplier $\Theta\in \clm(\arv(\cle),\wbg(\beta, \cle_*))$ such
that
\[
\bigvee_{\bk\in\Zp} {z}^{\bk}\clw= \Theta \arv(\cle).
\]
Moreover, if
\[
\clf= \{ \eta \in \cle : M_{\Theta}^* M_{\Theta} \eta = \eta \}
\subseteq \cle,
\]
then the wandering subspace $\clw$ and the multiplier $\Theta$ are related as follows (see Theorem 6.6, \cite{MB}):
\[
\clw =\Theta\clf,
\]
and
\[
\Theta|_{\arv(\clf)} \in \cls\clm(\arv(\clf), \wbg(\beta,
\cle_*))
\]
is a $K_{\beta}$-inner function. Now  we apply Theorem~\ref{thm-unique
factorization} to the $K_{\beta}$-inner function $\Theta|_{\arv(\clf)}$
and get that
\[
\Theta|_{\arv{\clf}} = \Psi_{\beta, \cle_*} \tilde{\Theta},
\]
where $\tilde{\Theta} \in \cls \clm ( \arv(\clf),
\arv(l^2(\Zp, \cle_*)))$ is the unique $K$-inner multiplier.
In particular,
\[
\tilde{\clw} :=\tilde{\Theta}\clf,
\]
is a wandering subspace for $M_z$ on $\arv(l^2(\mathbb{N}^n,
\cle_*))$, and so
\[
\clw = \Psi_{\beta, \cle_*} \tilde{\Theta} \clf = {\Psi_{\beta,
\cle_*}}\tilde{\clw}.
\]
This yields the following parametrization of a wandering subspace
for $M_z$ on $\wbg(\beta, \cle_*)$.

\begin{Theorem}
\label{thm-wandering subspace} If $\clw$ is a wandering subspace for
$M_z$ on $\wbg(\beta, \cle_*)$, then there exists a wandering
subspace $\tilde{\clw}$ for $M_z$ on $\arv(l^2( \Zp,
\cle_*))$ such that
\[
\clw = {\Psi_{\beta, \cle_*}}\tilde{\clw},
\]
where $\Psi_{\beta,\cle_*}$ is the universal multiplier.
\end{Theorem}

The above parametrizations of wandering subspaces is significantly
different from that of Eschmeier \cite{ES} and Olofsson \cite{AO1}.

\newsection{Factorizations and representations of Characteristic Functions}\label{sec -fact ch
fn}

We continue our study of pure $m$-hypercontractions by focusing on
the universal multipliers $\Psi_{\beta, \clh}$ and relate this idea
to the notion of the transfer functions on $\mathbb{B}^n$. Here we
follow the notation introduced in Section \ref{sec-4}.

Fix $m > 1$ and a weight sequence $\beta(m) = \{\beta_j(m)\}$ as
\[
\beta_j(m) = \binom{m+j-1}{j}^{-1},
\]
for all $j \in \mathbb{Z}_{+}$. Then the corresponding weight sequence
$\gamma(m) = \{\gamma_j(m)\}$ is given by
\[
\begin{split}
{\gamma_j(m)}^{-1} & = \frac{1}{\beta_j(m)} - \frac{1}{\beta_{j-1}(m)}
\\
& = \frac{(m+j-1)!}{j!(m-1)!} - \frac{(m+j-2)!}{(j-1)!(m-1)!}
\\
& = \frac{(m+j-2)!}{j!(m-2)!},
\end{split}
\]
that is
\[
\gamma_j(m) = \binom{m+j-2}{j}^{-1},
\]
for all $j \geq 1$. Then for a Hilbert space $\clf$, one finds that
\begin{equation}\label{eq-H2beta}
\wbg(\beta(m), \clf) = \bH_m(\mathbb{B}^n, \clf) \quad \mbox{and} \quad
\wbg(\gamma(m), \clf) = \bH_{m-1}(\mathbb{B}^n, \clf).
\end{equation}
Now let $T$ be a pure $m$-hypercontraction on $\clh$, and let
$(\cle, B, D)$ be a characteristic triple of $T$. Then $\Phi_T$, the
characteristic function of $T$ corresponding to $(\cle, B, D)$,
defined by
\[
\Phi_T(\z) =  \Big(\sum_{\bk \in \Zp} \sqrt{\rho_{m-1}(\bk)} D_{\bk}
z^{\bk} \Big) + D_{m,T^*} (I_{\clh} - Z T^*)^{-m} ZB \quad \quad (\z \in
\mathbb{B}^n),
\]
is a $\clb(\cle, \cld_{m, T^*})$-valued analytic function on
$\mathbb{B}^n$. Moreover
\[
\Phi_T \in \clm(\arv(\cle), \bH_m(\mathbb{B}^n, \cld_{m,T^*})),
\]
is a partially isometric multiplier (see Theorem \ref{thm-analytic
model}). Now, in view of \eqref{eq-H2beta}, Theorem \ref{factom}
implies that
\[
\Phi_T = \Psi_{\beta(m), \cld_{m, T^*}} \tilde{\Phi}_T,
\]
for some  contractive multiplier $\tilde{\Phi}_T \in
\clm(\arv(\cle), \arv(l^2(\Zp,\cld_{m,T^*})))$. Here
\[
\Psi_{\beta(m), \cld_{m, T^*}} \in \clm(\arv(l^2(\Zp,
\cld_{m,T^*})), \bH_m(\mathbb{B}^n, \cld_{m, T^*})),
\]
is the universal multiplier defined by
\[
\Psi_{\beta(m), \cld_{m, T^*}}(\z) = \Big[\cdots
\sqrt{\frac{\rho_1(\bk)}{\gamma_{|\bk|}(m)}} z^{\bk}I_{\cld_{m,T^*}}
\cdots \Big]_{\bk \in \Zp},
\]
for all $\z \in \mathbb{B}^n$. However, in our particular situation
\[
\gamma_{|\bk|}(m) = \binom{m + |\bk| - 2}{|\bk|},
\]
and hence
\[
\sqrt{\frac{\rho_1(\bk)}{\gamma_{|\bk|}(m)}} = \sqrt{\rho_{m-1}(\bk)},
\]
for all $\bk \in \Zp$. Then the universal multiplier is given by
\begin{equation}\label{eq-Psibetam}
\Psi_{\beta(m), \cld_{m, T^*}}(\z) = \Big[\cdots
\sqrt{\rho_{m-1}(\bk)} z^{\bk}I_{\cld_{m,T^*}} \cdots \Big]_{\bk \in
\Zp}.
\end{equation}

Now we proceed to compute an explicit representation of $\tilde{\Phi}_T$. To this end, we first recall that
\[
U = \begin{bmatrix}
T^*  &  B \\
\cmt  &  D
\end{bmatrix}
: \clh \oplus \cle \to \clh^n \oplus \lnt
\]
is unitary (see Theorem \ref{thm-ch triple}). We claim that
$\tilde{\Phi}_T$ is the transfer of the unitary $U$ (see \cite{AT}), that is,
\[
\tilde{\Phi}_T(\z) = D + \cmt (I_{\clh}-ZT^*)^{-1}ZB \quad \quad (\z
\in \mathbb{B}^n).
\]
%That is,  $\tilde{\Phi}_T$ is the transfer function corresponding to
%the unitary matrix $U$ (see \cite{AT}).
Indeed, first note that $\tilde{\Phi}_T \in
\clm(\arv(\cle), \arv(l^2(\Zp, \cld_{m,T^*})))$ (cf. \cite{AT}) and
\[
\begin{split}
\tilde{\Phi}_T(z) &= D + \cmt(I_{\clh} - ZT^*)^{-1}ZB
\\
& = D + \cmt \sum_{i=1}^n \big( \sum_{\bl \in \Zp} \rho_1(\bl) z^{\bl}T^{*\bl}
\big) z_i B_i
\\
& =  D + \sum_{i=1}^n  \sum_{\bl \in \Zp}\big(\rho_1(\bl) \cmt T^{*\bl} B_i
\big) z^{\bl + \bm{e}_i},
\end{split}
\]
for all $\z \in \mathbb{B}^n$, where
\[
B = \begin{bmatrix} B_1 \\ \vdots \\ B_n \end{bmatrix} : \cle \raro
\clh^n,
\]
and $\bm{e}_i \in \Zp$ has a $1$ in the $i$-th position and $0$
elsewhere, $i = 1, \ldots, n$. Then, by applying the matrix
representation of $\cmt$ (see \eqref{eq-cmt matrix}), we have
\[
\tilde{\Phi}_T(\z) = \begin{bmatrix} \vdots\\
\displaystyle  D_{\bk} + \sum_{i=1}^n \sum_{\bl \in \Zp}
\big(\sqrt{\rho_{m-1}(\bk)} \rho_1(\bl) \dmt T^{* (\bk + \bl)}B_i \big) z^{\bl + \bm{e}_i}\\
\vdots
\end{bmatrix}_{\bk \in \Zp},
\]
and consequently, by \eqref{eq-Psibetam}, we have
\[
\Psi_{\beta(m), \cld_{m, T^*}}(\z) \tilde{\Phi}_T (\z) =  \sum_{\bk
\in \Zp}
\Big( \sqrt{\rho_{m-1}(\bk)} D_{\bk} + \sum_{i=1}^n \sum_{\bl \in \Zp} \big(
{\rho_{m-1}(\bk)} \rho_1(\bl) \dmt T^{* (\bk + \bl)} B_i \big) z^{\bl +
\bm{e}_i} \Big) z^{\bk}.
\]
Also note that
\[
\begin{split}
D_{m,T^*} (I_{\clh} - Z T^*)^{-m} ZB & = D_{m,T^*} (I_{\clh} - Z
T^*)^{-(m-1)} (I_{\clh} - Z T^*)^{-1} ZB
\\
& = D_{m, T^*} \Big(\sum_{\bk \in \Zp} {\rho_{m-1}(\bk)} T^{* \bk}
z^{\bk} \Big) \Big(\sum_{\bl \in \Zp} \rho_1(\bl) T^{*\bl} z^{\bl} \Big) ZB
\\
& = D_{m, T^*} \sum_{i=1}^n z_i \Big(\sum_{\bk \in \Zp}
{\rho_{m-1}(\bk)} T^{* \bk} z^{\bk} \Big) \Big(\sum_{\bl \in \Zp}
\rho_1(\bl) T^{*\bl} z^{\bl} \Big) B_i
\\
& = \sum_{\bk, \bl \in \Zp} \sum_{i=1}^n \Big(\ {\rho_{m-1}(\bk)} \rho_1(\bl) D_{m, T^*} T^{* (\bk + \bl)} B_i \Big) z^{\bk + \bl + \bm{e}_i}.
\end{split}
\]
From this it readily follows that
\[
\begin{split}
\Psi_{\beta(m), \cld_{m, T^*}}(\z) \tilde{\Phi}_T (\z) & = \Big(\sum_{\bk \in \Zp} \sqrt{\rho_{m-1}(\bk)}
D_{\bk} z^{\bk} \Big) + D_{m,T^*} (I_{\clh} - Z T^*)^{-m} ZB
\\
& = \Phi_T(\z),
\end{split}
\]
for all $\z \in \mathbb{B}^n$. This leads to the following theorem
on explicit representation of $\tilde{\Phi}_T$:

\begin{Theorem}\label{thm-factchar}
Let $m \geq 1$, $T$ be a pure $m$-hypercontraction on a Hilbert space
$\clh$, and let $(\cle, B, D)$ be a characteristic triple of $T$. If
$\Phi_T$ is the characteristic function of $T$ corresponding to
$(\cle, B, D)$, then
\[
\Phi_T(\z) = \Psi_{\beta(m), \cld_{m, T^*}}(\z) \tilde{\Phi}_T (\z) \quad \quad (\z \in \mathbb{B}^n),
\]
where
\[
\tilde{\Phi}_T(z) = D + \cmt (I_{\clh} - ZT^*)^{-1} ZB  \quad \quad (\z \in \mathbb{B}^n),
\]
is the transfer function of the canonical unitary matrix
\[\begin{bmatrix}
T^*  &  B \\
\cmt  &  D
\end{bmatrix}: \clh \oplus \cle \to \clh^n \oplus \lnt
\]
corresponding to the characteristic triple $(\cle,B,D)$ of $T$, and

\[
\Psi_{\beta(m), \cld_{m, T^*}}(\z) = \begin{cases} \begin{bmatrix} \cdots &
\sqrt{\rho_{m-1}(\bk)} z^{\bk}I_{\cld_{m,T^*}} & \cdots \end{bmatrix}_{\bk
\in \Zp} & \mbox{if~} m \geq 2
\\
\begin{bmatrix}I_{\cld_{1, T^*}} & 0 & 0 & \cdots \end{bmatrix} & \mbox{if~} m=1, \end{cases}
\]
for all $\z \in \mathbb{B}^n$.
\end{Theorem}
\begin{proof}
It remains only to prove the special case $m=1$. Let $T$ be a pure $1$-hypercontraction, and let $(\cle, B, D)$ be a characteristic triple of $T$. Then \eqref{eq-cmt} implies that
\[
C_{1, T}(h) = ( D_{1, T^*} h,  0,  0,  \ldots) \quad \quad (h \in \clh),
\]
and so
\[
\Phi_T= D_0 + D_{1,T^*} (I_{\clh} - ZT^*)^{-1} ZB,
\]
for all $\z \in \mathbb{B}^n$. It now easily follows that
\[
\Phi_T=\Psi_{\beta(1), \cld_{1, T^*}}(\z) \tilde{\Phi}_T (\z) \quad \quad (\z \in \mathbb{B}^n).
\]
\end{proof}

We will refer
\[
\tilde{\Phi}_T \in \clm(\arv(\cle), \bH_m(l^2(\Zp, \cld_{m, T^*})),
\]
as the \textit{canonical transfer function} of $T$ corresponding to
the characteristic triple $(\cle, B, D)$.

\newsection{Hypercontractions and row-contractions}\label{sec 6}

The present theory of pure $m$-hypercontractions leads to many
interesting questions of analytic models, such as any possible
relationships between characteristic functions or canonical transfer
functions of $m'$-hypercontractions, $1 \leq m' < m$. Here we
address this issue. Also we compare the ideas of characteristic
functions of pure $m$-hypercontractions and characteristic functions
of pure row contractions.

First, we examine our construction of characteristic triples for pure
$1$-hypercontractions. Before doing so
we recall that the \textit{characteristic function} \cite{BES} of a
commuting row contraction (that is, $1$-hypercontraction) $T = (T_1,
\ldots, T_n)$ on a Hilbert space $\clh$ is the operator-valued
analytic function
\[
\Theta_T(\z) = [-T + D_{1, T^*}(I_{\clh} - ZT^*)^{-1}ZD_T]|_{\cld_T}
\in \clb(\cld_T, \cld_{1, T^*}) \quad \quad (\z \in \mathbb{B}^n),
\]
where $D_T = (I_{\clh^n} - T^* T)^{\frac{1}{2}}$ and $\cld_T =
\overline{\mbox{ran}} D_T$. Observe also that $\Theta_T$ is the
transfer function corresponding to the unitary (colligation) matrix
\[
\begin{bmatrix}
T^* & D_T\\
 D_{1, T^*} & -T
\end{bmatrix} : \clh \oplus \cld_T \to \clh^n \oplus \cld_{1, T^*},
\]
and $\Theta_T \in \clm(\arv(\cld_T), \arv(\cld_{1, T^*}))$ (cf. \cite{BES}).
In the following, we shall identify $\cld_{1,T^*}$ with
\[
\cld_{1, T^*} \oplus \{0\} \oplus \{0\} \oplus \cdots \subset l^2(\Zp,\cld_{1, T^*}),
\]
and view $\Theta_T\in \clm(\arv(\cld_T), \arv(\cld_{1, T^*} \oplus \{0\} \oplus \{0\} \oplus \cdots))$.

\begin{Theorem}\label{thm-row-hyper}
Let $T$ be a pure row contraction on $\clh$. Then there exists a
characteristic triple $(\cle, B, D)$ of $T$ such that $\cld_T
\subseteq \cle$ and
\[
\Theta_T(\z) = \tilde{\Phi}_T (\z)|_{\cld_T} \quad \quad (\z \in \mathbb{B}^n),
\]
where $\Tilde{\Phi}_T$, defined by
\[
\Tilde{\Phi}_T(\z) = D + C_{1, T} (I_{\clh} - Z T^*)^{-1} Z B\quad \quad (\z \in \mathbb{B}^n),
\]
and $\Theta_T$ are the canonical transfer function corresponding to $(\cle, B, D)$ and the characteristic function of $T$, respectively.
\end{Theorem}

\begin{proof}
Let $T$ be a pure row contraction (that is, pure $1$-hypercontraction). Set
\[
\cle:= \cld_T \oplus \tilde{\cle},
\]
where $\tilde{\cle} = l^2(\Zp,\cld_{1, T^*}) \ominus (\cld_{1, T^*}, 0,\cdots)$.
Define $B =[D_T,0]: \cle \to \clh^n$ by
\[
B(f, \{\alpha_{\bk}\}_{\bk \in \Zp}) = D_T f,
\]
and $D: \cle \to l^2(\Zp,\cld_{1, T^*})$ by
\[
\Big(D(f, \{\alpha_{\bk}\}_{\bk \in \Zp})\Big)(\bl) =
\begin{cases} -Tf & \mbox{if}~ \bl = \bm{0}
\\
\alpha_{\bl} & \mbox{otherwise} \end{cases}
\]
for all $f \in \cld_T$ and $\{\alpha_{\bk}\}_{\bk \in \Zp}
\in \tilde{\cle}$. Finally, define $C_{1, T} : \clh \raro l^2(\Zp,
\cld_{1, T^*})$ by
\[
C_{1,T} h = (D_{1,T^*}h,0,\cdots)\quad (h\in\clh).
\]
%where $\iota_0:\cld_{1, T^*} \to
%l^2(\Zp,\cld_{1, T^*})$ is the inclusion map
%\[
%\iota_0( D_{T^*}h)(\bk) =
%\begin{cases} D_{T^*}h & \mbox{if}~ \bk = \bm{0}
%\\
%0 & \mbox{otherwise}. \end{cases}
%\]
It is obvious that
\[
T T^* + C_{1, T}^* C_{1, T} = I_{\clh},
\]
and
\[
\begin{bmatrix}
T^* & B \\
  C_{1, T} & D
\end{bmatrix} : \clh \oplus \cle_1 \to \clh^n \oplus l^2(\Zp,\cld_{1, T^*}),
\]
is unitary, which implies that $(\cle, B, D)$ is a characteristic
triple of the $1$-hypercontraction $T$. The canonical transfer
function corresponding to $(\cle, B, D)$ is given by
\[
\Tilde{\Phi}_T(\z) = D + C_{1, T} (I_{\clh} - Z T^*)^{-1} Z B \quad \quad (\z \in \mathbb{B}^n).
\]
Then it readily follows that $\Theta_T(\z)=\tilde{\Phi}_T (\z)|_{\cld_T}$ for all
$\z \in \mathbb{B}^n$. This completes the proof of the theorem.
\end{proof}

We refer to the characteristic triple constructed above for
a pure $1$-hypercontraction as the {\it canonical characteristic triple}.

Now let $1 \leq m_1 < m_2$ and let $T$ be a pure
$m_2$-hypercontraction on $\clh$. Then $T$ is also a pure
$m_1$-hypercontraction. Suppose that $(\cle_i,B_i,D_i)$ is a
characteristic triple of the $m_i$-hypercontraction $T$, $i = 1, 2$. Then
\[
U_i=\begin{bmatrix}
T^* & B_{i} \\
C_{m_i, T} & D_{i}
\end{bmatrix} : \clh \oplus \cle_i \to \clh^n \oplus l^2(\Zp, \cld_{m_i,T^*}),
\]
is the unitary operator corresponding to the $m_i$-hypercontraction $T$, $i = 1, 2$. For simplicity of notation, we denote $\tilde{\Phi}_{T, m_i}$ the canonical transfer function corresponding to $(\cle_i,B_i,D_i)$, $i=1,2$. Since (see \eqref{eq-TCI})
\[
C_{m_i,T}^* C_{m_i,T} = I_{\clh} - TT^* \quad \quad (i=1,2),
\]
we have
\[
C_{m_1,T}^* C_{m_1,T} = C_{m_2,T}^* C_{m_2,T}.
\]
Also, according to \eqref{eq-matrix}, we have
\[
B_1B_1^* = B_2 B_2^*
\]
and
\[
D_iB_i^* = - C_{m_i, T} T,
\]
for $i = 1, 2$. It now follows by Douglas' range inclusion theorem that
\[
YC_{m_2,T} = C_{m_1,T} \quad \mbox{and} \quad XB_{1}^*=B_{2}^*,
\]
for some isometry $Y \in \clb(\overline{\mbox{ran} }\ C_{m_2,T}, l^2(\Zp,\cld_{m_1,T^*}))$
and unitary $X\in \clb(\overline{\mbox{ran}}\ B_{1}^*, \overline{\mbox{ran}}\ B_{2}^*)$. Thus
\[
D_1B_1^*=-C_{m_1,T}T=-YC_{m_2,T}T=YD_2B_2^*=YD_2X B_1^*,
\]
and so
\[
D_1|_{(\ker B_{1})^{\perp}}=YD_2X.
\]
This and the definition of $\tilde{\Phi}_{T, m_i}$, $i = 1, 2$, gives
\[
\begin{split}
\tilde{\Phi}_{T, m_1}(\z)|_{(\ker B_{1})^{\perp}}
&= \big[D_{1}+ C_{m_1, T}(I_{\clh} - ZT^*)^{-1}ZB_{1}\big]|_{(\ker B_{1})^{\perp}}
\\
& = YD_{2}X + YC_{m_2,T}(I_{\clh} - ZT^*)^{-1}Z B_{2} X
\\
& = Y \tilde{\Phi}_{T, m_2}(\z) X.
\end{split}
\]
This establishes the following relationship between canonical transfer functions:

\begin{Theorem}\label{tranre1}
Let $1 \leq m_1 < m_2$, $T$ be a pure $m_2$-hypercontraction on $\clh$,
and let $(\cle_i,B_i,D_i)$ be characteristic triple of the
$m_i$-hypercontraction $T$, $i = 1, 2$. Then there exist an isometry
$Y \in \clb(\overline{\mbox{ran} }\ C_{m_2,T},
l^2(\Zp,\cld_{m_1,T^*}))$ and a unitary $X\in \clb(\overline{\mbox{ran}}\
B_{1}^*, \overline{\mbox{ran}}\ B_{2}^*)$ such that
\[
\tilde{\Phi}_{T, m_1}(\z)|_{(\ker B_{1})^{\perp}}
= Y \tilde{\Phi}_{T, m_2}(\z) X \quad \quad (\z \in \mathbb{B}^n),
\]
where $\tilde{\Phi}_{T, m_i}$ is the canonical transfer function
corresponding to the characteristic triple $(\cle_i, B_i, D_i)$, $i = 1, 2$.
\end{Theorem}

\begin{Remark}
\label{re}
Let $\clf$, $\clf_*$, $\cle$ and $\cle_*$ be Hilbert spaces, and let
\[
U=\begin{bmatrix} A & B\\ C & D \end{bmatrix}:
\clf \oplus \cle\to \clf_*\oplus \cle_*,
\]
be a unitary. Suppose that $\Phi$ is the transfer function corresponding to $U$, that is
\[
\Phi(\z)= D+ C (I-ZA)^{-1}ZB \quad \quad (\z \in \mathbb{B}^n).
\]
Then $\Phi|_{(\ker B)^{\perp}}$ is the purely contractive part
of the contractive operator-valued analytic function
$\Phi$ on $\mathbb B^n$ in the sense of Sz.-Nagy and Foias \cite[Chapter V, Proposition 2.1]{NF}.
This follows from the observation that the maximal subspace of $\cle$
 where $D$ is an isometry is $\ker B$ and
 $D|_{\ker B}: \ker B \to \ker C^*$ is a unitary.
 Moreover, $\Phi|_{(\ker B)^{\perp}}$
 is the transfer function of the unitary
\[
\begin{bmatrix} A & B|_{(\ker B)^{\perp}}\\ C & D|_{(\ker B)^{\perp}} \end{bmatrix}:
\clf \oplus (\ker B)^{\perp}\to \clf_*\oplus \overline{\mbox{ran}} C.
\]
From this point of view, $\tilde{\Phi}_{T, m_1}(\z)|_{(\ker B_{1})^{\perp}}$, $\z \in \mathbb{B}^n$,
in the conclusion of Theorem ~\ref{tranre1} is the purely contractive part of $\tilde{\Phi}_{T, m_1}$. Moreover, $\mbox{ran} X =(\ker B_{2})^{\perp}$ implies that $Y\tilde{\Phi}_{T, m_1}(.)X$ coincides with
the purely contactive part of $\tilde{\Phi}_{T, m_1}$. Therefore Theorem ~\ref{tranre1} implies that
the purely contractive part of $\tilde{\Phi}_{T, m_1}$ coincides with the purely contractive part of $\tilde{\Phi}_{T, m_2}$.
\end{Remark}

We continue with the hypothesis that $T$ is a pure $m$-hypercontraction, $m > 1$. Let $(\cle_m, B_m, D_m)$ be a characteristic triple
of the pure $m$-hypercontraction $T$ and $\tilde{\Phi}_{T,m}$ be the
corresponding canonical transfer function. Since $T$ is also a
pure $1$-hypercontraction, consider the canonical characteristic triple $(\cle, B, D)$ of $T$
as obtained in the proof of Theorem \ref{thm-row-hyper}. Let $\tilde{\Phi}_T$ be
the canonical transfer function corresponding to $(\cle, B, D)$.
Then by Theorem ~\ref{tranre1},
\[
\tilde{\Phi}_{T}(\z)|_{(\ker B)^{\perp}} = Y \tilde{\Phi}_{T, m}(\z) X \quad \quad (\z \in \mathbb{B}^n),
\]
for some isometry $Y \in \clb(\overline{\text{ran}}\ C_{m,T}, l^2(\Zp,\cld_{1,T^*}))$ and unitary
$ X\in \clb(\overline{\mbox{ran}}\
B^*, \overline{\mbox{ran}}\ B_{m}^*)$. Moreover (see the construction of $B$ in the proof of Theorem \ref{thm-row-hyper})
\[
(\ker B)^{\perp} = \cld_T,
\]
and hence by Theorem \ref{thm-row-hyper}, it follows that
\[
\Theta_T(\z)=Y \tilde{\Phi}_{T,m}(\z)X \quad \quad (\z \in \mathbb{B}^n).
\]
Therefore, we have the following theorem:

\begin{Theorem}\label{transre2}
Let $m \geq 2$, $T$ be a pure $m$-hypercontraction
on $\clh$, and let $(\cle_m,B_m,D_m)$ be a characteristic triple of $T$.
Then there exist an isometry
$Y \in \clb(\overline{\text{ran}}\ C_{m,T}, l^2(\Zp,\cld_{1, T^*}))$
and a unitary $ X\in \clb(\cld_T, \overline{\mbox{ran}}\ B_{m}^*)$ such that
\[
\Theta_T(\z) =Y \tilde{\Phi}_{ T,m}(\z)X \quad \quad (\z \in \mathbb{B}^n),
\]
where $\Theta_T$ and $\tilde{\Phi}_{m, T}$
denote the characteristic function
of the row contraction $T$ and the canonical
transfer function of $T$ corresponding to the
characteristic triple $(\cle_m,B_m,D_m)$, respectively.
\end{Theorem}

\vspace{0.1in} \noindent\textbf{Acknowledgement:}
The first named author acknowledges IISc Bangalore and IIT Bombay for
 warm hospitality and his research is supported by the institute post-doctoral fellowship of IIT Bombay.
The research of the second named author is supported by  DST-INSPIRE Faculty
Fellowship No. DST/INSPIRE/04/2015/001094. The research of the third
named author is supported in part by the Mathematical Research
Impact Centric Support (MATRICS) grant, File No : MTR/2017/000522,
by the Science and Engineering Research Board (SERB), Department of
Science \& Technology (DST), Government of India, and NBHM (National
Board of Higher Mathematics, India) Research Grant NBHM/R.P.64/2014.

\bibliographystyle{amsplain}

\end{document}